\documentclass[11pt]{article}

\usepackage{amsfonts}
\usepackage{amsmath}
\usepackage{amsthm}
\usepackage{amssymb}
\usepackage{enumitem}
\usepackage{color}
\usepackage{url}
\usepackage{comment}
\usepackage{geometry}
\usepackage{bm}

\usepackage{tikz}
\tikzstyle{vertex}=[circle, draw, inner sep=2pt, fill=white]
\newcommand{\vertex}{\node[vertex]}

\newtheorem{thm}{Theorem}[section]
\newtheorem{cor}[thm]{Corollary}
\newtheorem{lem}[thm]{Lemma}
\newtheorem{prop}[thm]{Proposition}

\theoremstyle{definition}
\newtheorem{defn}[thm]{Definition}

\newtheorem{rem}[thm]{Remark}
\newtheorem{exa}[thm]{Example}
\numberwithin{equation}{section}
\numberwithin{thm}{section}

\newcommand{\R}{\mathbb{R}}
\newcommand{\N}{\mathbb{N}}
\renewcommand{\S}{\mathbb{S}}
\newcommand{\T}{{\bf T}}

\newcommand{\PP}{\mathbb{P}}
\newcommand{\QQ}{\mathbb{Q}}
\newcommand{\EE}{\mathbb{E}}

\newcommand{\expec}{\mathbb{E}}

\newcommand{\cadlag}{{c\`adl\`ag\ }}

\newcommand{\dbra}[1]{[\kern-0.15em[ #1 ]\kern-0.15em]}
\newcommand{\dbraco}[1]{[\kern-0.15em[ #1 [\kern-0.15em[}
\newcommand{\dbraoc}[1]{]\kern-0.15em] #1 ]\kern-0.15em]}
\newcommand{\dbraoo}[1]{]\kern-0.15em] #1 [\kern-0.15em[}

\newcommand{\Fcal}{\mathcal{F}}
\newcommand{\Gcal}{\mathcal{G}}
\newcommand{\Hcal}{\mathcal{H}}

\newcommand{\Mcal}{\mathcal{M}}
\newcommand{\Pcal}{\mathcal{P}}

\newcommand{\Scal}{\mathcal{S}}

\newcommand{\FF}{\mathbb{F}}
\newcommand{\GG}{\mathbb{G}}
\newcommand{\HH}{\mathbb{H}}

\DeclareMathOperator{\rk}{rank}
\DeclareMathOperator{\Diag}{Diag}
\DeclareMathOperator{\ext}{ext}
\DeclareMathOperator{\vspan}{span}

\DeclareMathOperator{\sgn}{sgn}
\newcommand{\lc}{[\![}
\newcommand{\rc}{]\!]}
\newcommand{\id}{{\bf I}}

\allowdisplaybreaks

\begin{document}

\title{Semi-static completeness and \\robust pricing by informed investors\footnote{We would like to thank Monique Jeanblanc for useful comments and discussions. We also thank the associate editor and two anonymous referees for their valuable comments.}}

\author{B. Acciaio\thanks{London School of Economics and Political Science, Dept. Statistics, 10 Houghton St, WC2A 2AE London, UK. Email: b.acciaio@lse.ac.uk} \and M. Larsson\thanks{ETH Zurich, Department of Mathematics, R\"amistrasse 101, CH-8092, Zurich, Switzerland. Email: martin.larsson@math.ethz.ch}}

\date{\today}

\maketitle

\begin{abstract}
We consider a continuous-time financial market that consists of securities available for dynamic trading, and securities only available for static trading. We work in a robust framework where a set of non-dominated models is given. The concept of semi-static completeness is introduced: it corresponds to having exact replication by means of semi-static strategies.  We show that semi-static completeness is equivalent to an extremality property, and give a characterization of the induced filtration structure. Furthermore, we consider investors with additional information and, for specific types of extra information, we characterize the models that are semi-statically complete for the informed investors. Finally, we provide some examples where robust pricing for informed and uninformed agents can be done over semi-statically complete models.\\[2ex]

\noindent{\textbf {Keywords:} Semi-static completeness; Robust finance; Extreme points; Filtration enlargement; Informed pricing.}
\\[2ex]
\noindent{\textbf {MSC2010 subject classifications:} 91G10, 60G44.}
\end{abstract}

\section{Introduction}

We consider a continuous-time financial market with two types of traded instruments: securities that can be traded dynamically over time, and securities that can be traded only at time zero and then held until maturity. In this setting we study replicability of contingent claims. The central notion of the present paper is {\em semi-static completeness}, which refers to the possibility of perfectly replicating any contingent claim by trading in the two sets of available securities.

Markets with semi-static trading opportunities are typically considered in the robust and model-free paradigm in mathematical finance. According to this paradigm, rather than committing to one particular model (probability measure), one considers a whole class of models that may be non-dominated in the sense of absolute continuity of probability measures. The goal is to account for model uncertainty via a worst-case analysis over models that are deemed plausible. A large and growing literature has developed around the following informally stated problem: In a suitable framework, establish the duality formula
\begin{equation} \label{eq:intro}
\sup_{\QQ\in\Mcal} \EE_\QQ[ \Phi ] = \inf \left\{ x\in\R\colon \begin{array}{l}\text{$\Phi$ can be super-replicated by semi-static trading}\\\text{starting from initial capital $x$ } \end{array}\right\},
\end{equation}
where $\Mcal$ denotes an appropriate set of martingale measures, and $\Phi$ is a contingent claim. The crucial feature is that super-replication is required $\Mcal$-quasi-surely. Following the seminal paper by Hobson~\cite{Ho98}, this problem has been addressed by many authors; for a survey, see \cite{Ho11} and \cite{To14}. In particular, \eqref{eq:intro} has been proved in a variety of settings; see e.g.~\cite{BHP13,ABPS13,Nu14,BN15,BZ14,FH14} in discrete time, and \cite{GHT14,DS14a,DS14b,BCH14,BBKN14,BNT15,HO15,GTT15,BCHPP15} in continuous time, among many others. 

How is this related to semi-static completeness? First, in some situations the set $\Mcal$ turns out to be convex and compact; see Section~\ref{sect:ex} for some examples. It is then often possible to reduce the left-hand side of~\eqref{eq:intro} to a maximization problem over the extreme points of $\Mcal$. Our first main result extends the classical Jacod-Yor theorem on extreme points and martingale representation to the setting of semi-static trading and non-dominated models. We show that semi-static completeness is equivalent to an extremality property of the model under consideration; see Theorem~\ref{T:ext}. This result is related to work by Campi and Martini \cite{CM15}, who study extremality in a two period model with countable state space.

Second, it is clear that both sides of~\eqref{eq:intro} depend on the underlying filtration, which models the information set available to investors in the market. Indeed, as the filtration becomes larger, the martingale property becomes more restrictive. This reduces the set $\Mcal$ and decreases the left-hand side of~\eqref{eq:intro}. Similarly, a larger information set increases the number of available trading strategies, which decreases the right-hand side of~\eqref{eq:intro}. Under structural assumptions on the filtrations involved, we characterize those models that are semi-statically complete for an informed agent with access to a larger filtration; see Theorem~\ref{T:finiteljump}. Under a semi-statically complete model, the informed agent will find that the larger filtration coincides with the original one up to nullsets. The message that emerges is that the only pricing measures that remain relevant for the informed agent are those under which the two filtrations are indistinguishable.

Third, to leverage these observations, the probabilistic structure of semi-statically complete models needs to be clarified. Indeed, given the long history of the notion of completeness in mathematical finance, it seems natural to study this notion in the context of semi-static trading. Assuming that the dynamically traded securities have continuous price paths, we provide a full characterization of semi-static completeness in terms of dynamic completeness; see Theorem~\ref{T:S cont char}. Dynamically complete models have been studied extensively in mathematical finance, and their structure is well understood. Our characterization theorem can thus be viewed as a recipe by which semi-statically complete models can be constructed. Let us mention that the theorem relies crucially on a structure that we refer to as an atomic tree, which is related to the underlying filtration; see Definition~\ref{D:atomic tree}.

The rest of the paper is organized as follows. In Section~\ref{Sect:setup} the mathematical setup is given, in particular the definition of semi-static completeness. Section~\ref{Sect:JY} contains the generalization of the Jacod-Yor theorem to the semi-static setting. Section~\ref{Sect:tree} is devoted to the characterization of semi-statically complete models. In Section~\ref{S:filter} semi-static completeness is studied in relation to changes of filtration.  In Section~\ref{sect:ex} we provides some examples where the sets of martingale measures are compact, thus the results of Section~\ref{S:filter} can be used to compare robust pricing for agents with different sets of information.
Finally, the Appendix contains the proof of the main result in Section~\ref{Sect:tree}, and an extension of the Jeulin-Yor theorem needed in Section~\ref{S:filter}.

\section{Setup and notation}\label{Sect:setup}

Let $T\in(0,\infty)$ be a finite time horizon and fix a filtered measurable space $(\Omega,\Fcal,\FF)$ whose filtration $\FF=(\Fcal_t)_{0\le t\le T}$ is right-continuous. No probability measure is given a priori. Consider a financial market consisting of two types of securities: a risk-fee savings account and $m$ risky assets available for dynamic trading, as well as $n$ securities available for static trading, meaning that they can only be bought or sold at time zero and must be held until time $T$. The prices of the dynamically traded risky assets, discounted by the savings account, are modeled by a \cadlag $\FF$-adapted process $S=(S_t)_{0\le t\le T}$, $S_t=(S^1_t,\ldots,S^m_t)$. We set $S_0=0$ without loss of generality. The discounted time $T$ payoffs of the statically traded securities are represented by a set $\Psi=\{\psi_1,\ldots,\psi_n\}$ of $\Fcal_T$-measurable random variables. Without loss of generality, we fix the price at time zero of each statically traded security to be zero. Note that discrete time models are included in this framework; see Example~\ref{ex:d}.

\paragraph{Calibrated martingale measures.}

Let $\Pcal=\Pcal(\FF)$ be a fixed set of probability measures on $\Fcal_T$ (\emph{priors}). The role of the set $\Pcal$ is to identify which events are deemed relevant and which are not. Adopting the notation in \cite{BN15}, we write $\QQ\lll\Pcal$ when $\QQ$ is a probability measure on $\Fcal_T$ such that $\QQ\ll\PP$ for some $\PP\in\Pcal$. We consider the following set of {\em calibrated martingale measures}, under which the price process $S$ is a martingale and the statically traded securities are priced correctly:
\begin{align*}
\Mcal(\FF) &= \left\{\QQ\lll\Pcal \colon \begin{array}{l} S \text{ is an $\FF$-martingale under $\QQ$},\\[1ex] \EE_\QQ[\psi_i]=0\text{ and } \EE_\QQ[\psi_i^2]<\infty \text{ for all $i$} \end{array} \right\}.
\end{align*}
Thus $\Mcal(\FF)$ is the set of martingale measures emerging in a robust framework, corresponding to the set $\Pcal$ of priors; see e.g.~\cite{GHT14} and \cite{BN15}. The martingale and calibration conditions are standard in the robust (non-dominated) setting. In addition, we require square integrability of the statically traded securities. It is common in the literature to impose some integrability condition on the asset's terminal distribution; for instance, \cite{ABPS13} require the existence of some superlinear moment, and \cite{DS14a,DS14b} and \cite{HO15} assume $L^p$-integrability for some $p>1$. Here we only require $S$ to be a martingale, but will impose $L^2$ integrability on trading strategies. This will allow us to use tools which are particular to the $L^2$ structure, such as orthogonal projections.

Note that, if $\Pcal$ is chosen to be the set of {\em all} probability measures on $\Fcal_T$, then, modulo integrability conditions, $\Mcal(\FF)$ is the usual set of martingale measures considered in the model-independent framework; see e.g.~\cite{BHP13} and \cite{DS14a}.

\paragraph{Semi-static completeness and extreme points.}
We now define semi-static completeness, which is the key notion of the present paper. For a probability measure $\QQ\in\Mcal(\FF)$, we denote by $\Hcal^2(\FF,\QQ)$ the set of square integrable martingales, and for a martingale $M=(M^1,\ldots,M^d)$ we let $L^2(M,\FF,\QQ)$ denote the set of $M$-integrable processes $H$ such that $H\cdot M\in\Hcal^2(\FF,\QQ)$. Here the usual vector stochastic integral as in \cite{JS03} or \cite{SC02} is used. Following \cite{DM80}, we use the convention $\Delta M_0=M_0$ and $(H\cdot M)_0=\sum_{i=1}^d H_0^iM_0^i$. This does not affect integrals with respect to the price process $S$ since $S_0=0$, but will be important when we integrate with respect to other martingales. Note that $\FF$ is not augmented with the $\QQ$-nullsets. If the filtration $\FF$ and probability measure $\QQ$ are clear from the context, we often drop them from the notation.

\begin{defn}
We say that {\em semi-static completeness holds under $\QQ\in\Mcal(\FF)$} if any $X\in L^2(\Fcal_T,\QQ)$ can be represented as
\[
X = x + a_1\psi_1 + \cdots + a_n\psi_n + (H\cdot S)_T \qquad \QQ\text{-a.s.}
\]
for some $x,a_1,\ldots,a_n\in\R$ and $H\in L^2(S,\FF,\QQ)$.
\end{defn}

Thus semi-static completeness means that any square-integrable payoff can be replicated using a semi-static strategy. In the absence of statically traded securities, semi-static completeness corresponds exactly to the usual predictable representation property in an $L^2$ setting. The main result of Section~\ref{Sect:JY} relates semi-static completeness to extremality of measures.

\begin{defn}
An element $\QQ\in\Mcal(\FF)$ is called an \emph{extreme point} if $\QQ=\lambda\QQ^1+(1-\lambda)\QQ^2$ with $\QQ^1,\QQ^2\in\Mcal(\FF)$ and $0<\lambda<1$ implies $\QQ^1=\QQ^2=\QQ$. The set of all extreme points of $\Mcal(\FF)$ is denoted by $\ext \Mcal(\FF)$.
\end{defn}

If $\Mcal(\FF)$ is convex, and the space of probability measures is endowed with a topology under which $\Mcal(\FF)$ is compact, the Krein-Milman theorem implies that $\Mcal(\FF)$ is the closed convex hull of its extreme points. For any payoff~$\Phi$ such that $\QQ\mapsto\EE_\QQ[\Phi]$ is continuous, one can then compute its robust super-hedging price over the set of extreme points:
\begin{equation}\label{eq:supexp}
\sup_{\QQ\in\Mcal(\FF)}\expec_\QQ[\Phi]=
\sup_{\QQ\in\ext\Mcal(\FF)}\expec_\QQ[\Phi].
\end{equation}
In Section~\ref{sect:ex} we provide two examples where $\Mcal(\FF)$ is compact. Note, however, that compactness of $\Mcal(\FF)$ is {\em not} assumed in any of our subsequent results; the above remarks merely serve as one motivation for studying its extreme points.

\begin{rem}\label{rem:rssc}
Semi-static completeness is a property of a given model $\QQ\in\Mcal(\FF)$. One could also think of a \emph{robust} notion, where replication is possible under all models in $\Mcal(\FF)$ simultaneously. However, in view of Theorem~\ref{T:ext}, this is equivalent to $\Mcal(\FF)$ being a singleton. Apart from the classical case $\Pcal=\{\PP\}$, such a situation seems too restrictive to be of much interest. Furthermore, Theorem~\ref{T:ext} shows that our notion of semi-static completeness is the right one in order to characterize $\ext\Mcal(\FF)$, which is a robust object in that it does not depend on any specific choice of reference measure.
\end{rem}

\paragraph{Stable subspaces and orthogonality.}
For a fixed probability measure~$\QQ\in\Mcal(\FF)$ and a possibly multi-dimensional martingale $M=(M^1,\ldots,M^d)$, we denote by $\Scal(M)$ the closed subspace of $\Hcal^2$ given by
\[
\Scal(M) = \{ H\cdot M \colon H\in L^2(M,\FF,\QQ)\}.
\]
If $M$ is square-integrable, then $\Scal(M)$ is the smallest closed subspace of $\Hcal^2$ that contains $M^1,\ldots,M^d$ and is stable under stopping. This is usually taken as the definition of $\Scal(M)$, which however is inconvenient for us, since in particular the prices process $S$ need not be square-integrable. In this paper, we only need the fact that $\Scal(M)$ is closed in $\Hcal^2$ and stable under stopping.

Recall that there are two notions of orthogonality for square-integrable martingales: $M$ and $N$ are {\em weakly orthogonal} if $\EE_\QQ[M_TN_T]=0$, that is, if $M$ and $N$ are orthogonal with respect to the inner product on the Hilbert space $\Hcal^2$. They are {\em strongly orthogonal} if $MN$ is a martingale. To simplify notation, we often identify square-integrable martingales with their final values. In particular, we then view $\Scal(M)$ as a subspace of $L^2(\Fcal_T)$.

\section{A semi-static Jacod-Yor theorem}\label{Sect:JY}

The classical Jacod-Yor theorem relates the predictable representation property of a process~$X$ to the extreme points of the set of martingale measures for~$X$; see  \cite{JY77}. The main result of the present section is an analog of the Jacod-Yor theorem in the context of semi-static hedging. It states that the extreme points of $\Mcal(\FF)$ exactly correspond to those models which are semi-statically complete.

\begin{thm}\label{T:ext}
Let $\QQ\in\Mcal(\FF)$. The following conditions are equivalent:
\begin{enumerate}
\item\label{T:ext:1} $\QQ \in \ext \Mcal(\FF)$.
\item\label{T:ext:3} Semi-static completeness holds under $\QQ$.
\end{enumerate}
\end{thm}

The proof of Theorem~\ref{T:ext} requires two auxiliary results. The first one shows that the set of outcomes of semi-static strategies as a subset of $L^2(\Fcal_T)$ is closed with respect to convergence in $L^1$.

\begin{lem} \label{L:Wclosed}
Let $\QQ\in\Mcal(\FF)$. The set of outcomes of semi-static strategies,
\[
W = \{a_0 + a_1\psi_1 + \cdots + a_n\psi_n + (H\cdot S)_T \colon a_0,\ldots,a_n\in\R \text{ and } H\in L^2(S)\},
\]
is closed in the following sense: if $(X_k)\subseteq W$ satisfies $X_k\to X$ in $L^1$ for some $X\in L^2(\Fcal_T)$, then $X\in W$.
\end{lem}

\begin{proof}
We proceed by induction on $n$. Suppose the result is known to be true with $n$ replaced by $n-1$, and let $W'$ be defined as $W$ with $n$ replaced by $n-1$. Let $(X_k)\subseteq W$ be a sequence satisfying $X_k\to X$ in $L^1$ for some $X\in L^2(\Fcal_T)$. Then $X_k=X_k' + a_k \psi_n$ for some $X_k'\in W'$ and $a_k\in\R$. If $\psi_n$ lies in the $L^1$-closure of $W'$, then the induction hypothesis yields $\psi_n\in W'$, so that in fact $(X_k)\subseteq W'$ and hence $X\in W'\subseteq W$ by another application of the induction hypothesis. We may thus suppose that $\psi_n$ does not lie in the $L^1$-closure of $W'$. Then by the Hahn-Banach theorem there exists a continuous linear functional $F$ on $L^1(\Fcal_T)$ that vanishes on $W'$ and satisfies $F(\psi_n)=1$. Thus $a_k=F(X_k)\to F(X)$, whence $X'_k \to X-F(X)\psi_n$ in $L^1$. The induction assumption then yields $X-F(X)\psi_n\in W'$ and hence $X\in W$, as desired.

It remains to prove the result for $n=0$. This follows immediately from the following result by \cite{Yo78}; see Theorem~15.4.7 in~\cite{DS06} for a formulation that covers the multidimensional case:

{\it Let $H^k$, $k\ge1$, be $S$-integrable processes such that $H^k\cdot S$ is a martingale for each $n$, and suppose $(H^k\cdot S)_T\to X$ in $L^1$ for some random variable $X$. Then there exists an $S$-integrable process $H$ such that $H\cdot S$ is a martingale with $(H\cdot S)_T=X$ a.s.}

Since in our case $X\in L^2(\Fcal_T)$, we additionally obtain $H\in L^2(S)$ and hence $X\in W$. This completes the proof.
\end{proof}

\begin{proof}[Proof of Theorem~\ref{T:ext}]
\ref{T:ext:1} $\Longrightarrow$ \ref{T:ext:3}:
By Lemma~\ref{L:Wclosed} it suffices to show that the linear span of $\Scal(S)$ and $1,\psi_1,\ldots,\psi_n$, which we denote by $W$, is dense in $L^1(\Fcal_T)$. Indeed, suppose this has been proved. Then for any $X\in L^2(\Fcal_T)$ we can find a sequence $(X_k)\subseteq W$ with $X_k\to X$ in $L^1$, whence $X\in W$ by Lemma~\ref{L:Wclosed}.

It remains to prove that $W$ is dense in $L^1(\Fcal_T)$. This follows from an application of Douglas's theorem; see \cite{Do64}. For completeness we provide the short argument. By the Hahn-Banach theorem it suffices to show that $Z=0$ for any random variable $Z\in L^\infty(\Fcal_T)$ such that $\EE_\QQ[YZ]=0$ for all $Y\in W$. Pick any such $Z$. By scaling, we may assume $|Z|\le 1/2$. Define probability measures $\QQ^+$ and $\QQ^-$ by
\[
d\QQ^{\pm} = (1 \pm Z)d\QQ.
\]
Since the Radon-Nikodym derivatives lie in $[1/2,3/2]$, a random variable is square integrable under $\QQ$ if and only if it is square integrable under $\QQ^{\pm}$. Moreover, we have
\[
\EE_{\QQ^{\pm}}[\psi_i]=\EE_\QQ[\psi_i] \pm\EE_\QQ[Z\psi_i]=0
\]
for all $i=1,\ldots,n$. Similarly, $\EE_{\QQ^{\pm}}[(H\cdot S)_T]=\EE_\QQ[(H\cdot S)_T]=0$ for all simple predictable bounded integrands~$H$.
Since also $\QQ^{\pm}\ll\QQ$, we have  $\QQ^{\pm}\in\Mcal(\FF)$. Now, since $\QQ=\frac{1}{2}\QQ^++\frac{1}{2}\QQ^-$ and $\QQ$ is an extreme point, it follows that $\QQ^+=\QQ^-=\QQ$, whence $Z=0$ as required.

\ref{T:ext:3} $\Longrightarrow$ \ref{T:ext:1}:
Suppose $\QQ=\lambda\QQ^1+(1-\lambda)\QQ^2$ for some $\QQ^1,\QQ^2\in\Mcal(\FF)$ and $\lambda\in(0,1)$. Then $\QQ^1\ll\QQ$, so we may define $Z=\frac{d\QQ^1}{d\QQ}$. For any $X=a_0+a_1\psi_1+\cdots+a_n\psi_n+(H\cdot S)_T\in W$ we then have
\[
\EE_\QQ[(Z-1) X] = \EE_{\QQ^1}[X] - \EE_{\QQ}[X]=a_0-a_0=0.
\]
Since $W$ is all of $L^2(\Fcal_T,\QQ)$ by assumption, it follows that $Z=1$ and hence $\QQ^1=\QQ^2=\QQ$. Thus $\QQ$ is an extreme point.
\end{proof}

\begin{rem}
Theorem~\ref{T:ext} could also be stated and proved in the $L^1$ setting, where the $\psi_i$ are only assumed integrable, and semi-static completeness is defined using integrands $H$ such that $H\cdot S$ is a martingale. This $L^1$ version of Theorem~\ref{T:ext} is easily proved by observing that the correspondingly modified set $W$ in~Lemma~\ref{L:Wclosed} is closed in $L^1$ (the proof is essentially the same). Since subsequent developments rely rather strongly on the Hilbert space structure of $L^2$, we opt to work in the $L^2$ setting throughout the paper in order to maintain consistency.
\end{rem}

In the classical setting of dynamic hedging without static components, there exists a wide range of complete models. For instance, dynamic completeness holds as soon as the price process is a strong solution to a possibly path-dependent stochastic differential equation of the form $dS_t = \sigma(t,S_u:u\le t)dW_t$, where $W$ is Brownian motion and $\sigma$ never vanishes. One may thus wonder whether, in the semi-static setting, there is any reason to expect complete models to exhibit further structural properties.

We now indicate why one might expect this to be the case. To this end, consider some $\QQ\in\Mcal(\FF)$ under which semi-static completeness holds, and suppose temporarily that $\Psi=\{\psi\}$ contains one single element. Consider the non-hedgeable part $\psi - \pi(\psi)$ of $\psi$, where $\pi$ denotes the orthogonal projection onto the closed subspace $\{X_T:X\in \Scal(S)\}\subseteq L^2(\Fcal_T)$. Let $M$ be the square-integrable martingale generated by this non-hedgeable part, $M_t=\EE_\QQ[\psi-\pi(\psi)\mid\Fcal_t]$. Then $M$ is weakly, hence strongly, orthogonal to $\Scal(S)$. It follows that $H\cdot M$ is (weakly and strongly) orthogonal to $\Scal(S)$ for any $H\in L^2(M)$, so that, by semi-static completeness, $H\cdot M$ lies in $\vspan\{M\}$, the linear span of $M$. Furthermore, the inclusion $\vspan\{M\}\subseteq \Scal(M)$ holds due to our convention regarding the time-zero value of stochastic integrals. Consequently,
\[
\Scal(M) = \vspan\{M\}.
\]
Thus the set of stochastic integrals with respect to $M$ is one-dimensional, which obviously imposes severe restrictions on the behavior of $M$; see Proposition~\ref{P:SM=spanM} for a precise statement in a multidimensional setting. Developing these observations further, one is led to a description of the behavior of semi-statically complete models in terms of dynamically complete models. This is the topic of the next section.

\section{Semi-static completeness for continuous price processes}\label{Sect:tree}

The goal of this section is to characterize the behavior of semi-statically complete models with continuous price processes. We consider a probability measure~$\QQ\in\Mcal(\FF)$ that will remain fixed throughout this section. Relations between random quantities are understood in the $\QQ$-almost sure sense, and to simplify notation we drop the subscript $\QQ$ and write $\EE[\,\cdot\,]=\EE_\QQ[\,\cdot\,]$.

A key notion needed in the characterization theorem is that of an atomic tree. For a set $A\in\Fcal_T$ we denote by $t(A)$ the first time $A$ becomes measurable,
\[
t(A) = \inf\{ t\in[0,T] \colon A\in\Fcal_t\}.
\]
Note that $A\in\Fcal_{t(A)}$ by right-continuity of $\FF$, and that $A\notin\Fcal_{t(A)-}$ if $t(A)>0$. Recall that $A$ is an {\em atom} of $\Fcal_t$ if $A\in\Fcal_t$ and $\QQ(B)$ equals zero or $\QQ(A)$ whenever $B\in\Fcal_t$, $B\subseteq A$.

\begin{defn} \label{D:atomic tree}
An {\em atomic tree} is a finite collection $\T$ of events in $\Fcal_T$ satisfying the following properties:
\begin{enumerate}
\item every $A\in\T$ is a non-null atom of $\Fcal_{t(A)}$;
\item for every $A,A'\in\T$ such that $t(A)< t(A')$, either $A\supseteq A'$ or $A\cap A'=\emptyset$;
\item for every $A,A'\in\T$ such that $A\supsetneq A'$, $\QQ(A\setminus A')>0$.
\end{enumerate}
\end{defn}

In order to discuss atomic trees $\T$ it is useful to have the following terminology: an element $A'\in\T$ is called a {\em child} of another element $A\in\T$ if $A'\subsetneq A$ and there is no $A''\in\T$ such that $A'\subsetneq A''\subsetneq A$. Moreover, $A\in\T$ is called a {\em leaf} if it has no children, or equivalently, if there is no $A'\in\T$ such that $A'\subsetneq A$. It is clear that $\T$ admits a natural tree structure obtained by connecting each element $A$ to its children. In particular, the above notion of a leaf coincides with the usual graph-theoretic notion. See Figure~\ref{F:tree} for an illustration. A natural measure of the size of an atomic tree $\T$ is the number of paths through the tree, or equivalently the number of leaves. We refer to this quantity as the {\em dimension} of $\T$,
\[
\dim\T = \text{number of leaves in $\T$}.
\]
This terminology is motivated by the fact that the space of functions defined on the set of paths through the tree has dimension $\dim\T$; see also Remark~\ref{R:atomic tree}\ref{R:atomic tree:4} below.

We now introduce a natural non-degeneracy condition on atomic trees.

\begin{defn}
An atomic tree $\T$ is called {\em full} if its leaves form a partition of $\Omega$ (up to nullsets), and if $A$ is an atom of $\Fcal_{t(A')-}$ whenever $A'$ is a child of $A$.
\end{defn}

\begin{figure}[h]
\begin{center}
\begin{tikzpicture}
    \node[left] () at (-0.2, 0) {$\Omega$};
    \node[left] () at (7-0.2, 2.7) {$A_7$};
    \node[left] () at (7-0.2, 1.7) {$A_6$};
    \node[left] () at (4.5-0.2, 0.7) {$A_5$};
    \node[left] () at (4.5-0.2, -0.7) {$A_4$};
    \node[left] () at (2-0.2, -2.2) {$A_1$};
    \node[left] () at (2+0.7, -0.4) {$A_2$};
    \node[left] () at (2-0.2, 2.2) {$A_3$};
    
    \vertex (0) at (0, 0) {};
    \coordinate (10) at (1.7, 0) {};
    \vertex (11) at (2, 2.2) {}; 
    \vertex (12) at (2, 0) {}; 
    \vertex (13) at (2, -2.2) {}; 
    \coordinate (110) at (6.8, 2.2) {}; 
    \vertex (111) at (7, 2.7) {}; 
    \vertex (112) at (7, 1.7) {}; 
    \coordinate (120) at (4.3, 0) {}; 
    \vertex (121) at (4.5, 0.7) {}; 
    \vertex (122) at (4.5, -0.7) {}; 
    \coordinate (1210) at (8, 0.7) {}; 

    \begin{scope}[every path/.style={-}]
       \draw (0) -- (10) ;
       \draw (10) -- (11) ;
       \draw (10) -- (12) ;
       \draw (10) -- (13) ;

       \draw (11) -- (110) ;
       \draw (110) -- (111) ;
       \draw (110) -- (112) ;

       \draw (12) -- (120) ;
       \draw (120) -- (121) ;
       \draw (120) -- (122) ;
    \end{scope}
    
    \draw[dotted] (0,-3) -- (0,3);
    \draw[dotted] (2,-3) -- (2,3);
    \draw[dotted] (4.5,-3) -- (4.5,3);
    \draw[dotted] (7,-3) -- (7,3);
    \draw[dotted] (8,-3) -- (8,3);
    \draw[thick,->] (0,-3) -- (0,-3) node[below] {$0$} -- (2,-3) node[below] {$t_1$}  -- (4.5,-3) node[below] {$t_2$}  -- (7,-3) node[below] {$t_3$}  -- (8,-3) node[below] {$T$} -- (9,-3);
\end{tikzpicture}
\end{center}
\caption{Schematic illustration of the filtration $\FF$ containing a full atomic tree $\T$. Each circle denotes an event $A\in\T$. In particular, the leaves of $\T$ are $\{A_1,A_4,A_5,A_6,A_7\}$. The lines denote relations between the elements of $\T$; for example, $A_4$ and $A_5$ are children of $A_2$, which in turn is a child of $\Omega$. Furthermore, we have $t_1=t(A_1)$, $t_2=t(A_4)=t(A_5)$, and $t_3=t(A_6)=t(A_7)$, and thus $\zeta(\T)=t_1\bm 1_{A_1}+t_2\bm 1_{A_2}+t_3\bm 1_{A_3}$.}
\label{F:tree}
\end{figure}
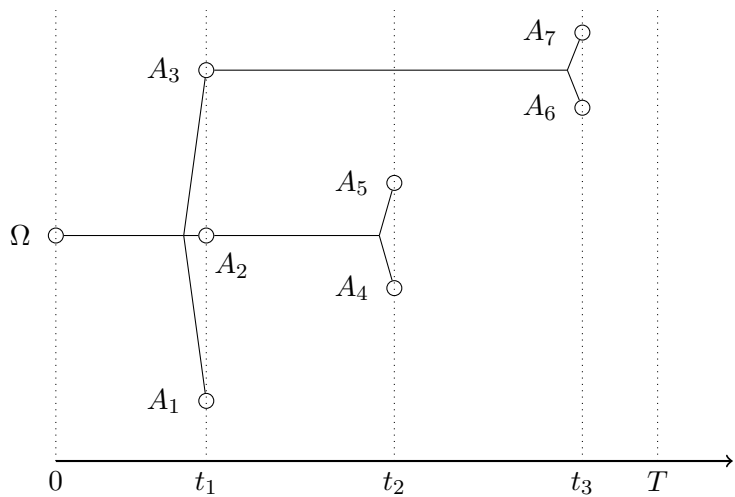

The following remark collects some basic properties and observations regarding full atomic trees that are immediate consequences of the above definitions.

\begin{rem} \label{R:atomic tree}
Let $\T$ be a full atomic tree.
\begin{enumerate}
\item Each $A\in\T$ that is not a leaf is the union of its children up to nullsets. Moreover, if $A'$ and $A''$ are children of $A$, then $t(A')=t(A'')$.
\item\label{R:atomic tree:3} Since $\T$ is a collection of elements of $\Fcal$, the sigma-algebra $\sigma(\T)\subseteq\Fcal$ is well-defined. Furthermore, up to nullsets, $\sigma(\T)=\sigma(A\in\T\colon\text{$A$ is a leaf})$. Since also the leaves form a partition of $\Omega$,
\begin{equation} \label{eq:EXsigma(T)}
\EE[X\mid\sigma(\T)] = \sum_A \frac{\EE[X\bm 1_A]}{\QQ(A)}\bm 1_A
\end{equation}
holds for any $X\in L^1(\Fcal_T)$, where the sum extends over all leaves $A\in\T$. Furthermore, $\sigma(\T)$ can alternatively be described up to nullsets as
\begin{equation} \label{eq:zeta(T)}
\sigma(\T) = \Fcal_{\zeta(\T)}  \qquad\text{with}\qquad \zeta(\T) = \sum_A t(A)\bm 1_A,
\end{equation}
where again the sum extends over the leaves of $\T$. Note that $\zeta(\T)$ is a stopping time bounded above by~$T$, which can be thought of as the ``end of the tree''.
\item\label{R:atomic tree:4} In view of~\eqref{eq:EXsigma(T)}, the dimension of $L^2(\sigma(\T))$ equals the number of leaves in $\T$. This motivates the definition of~$\dim\T$.
\end{enumerate}
\end{rem}

Finally, the following restricted notion of (dynamic) completeness is needed.

\begin{defn} \label{D:restrcompl}
Given $t\in[0,T]$ and $A\in\Fcal_t$, we say that $S$ is {\em complete on $A\times[t,T]$} if any $X\in L^2(\Fcal_T)$ can be replicated on $A$ by dynamic trading over $[t,T]$; that is, if
\[
X = x + (H\cdot S)_T \quad \text{on} \quad A
\]
holds almost surely for some $x\in\R$ and some $H\in L^2(S)$ with $H=0$ on $\lc0,t\rc$.
\end{defn}

\begin{rem}
If $S$ is complete on $A\times[t,T]$, then $A$ is necessarily an atom of~$\Fcal_t$. Indeed, if the arbitrarily chosen random variable $X$ in Definition~\ref{D:restrcompl} is $\Fcal_t$-measurable, then it is necessarily almost surely constant on~$A$ since $X\bm 1_A=\EE[x+(H\cdot X)_T\mid\Fcal_t]\bm 1_A=x\bm 1_A$.
\end{rem}

We can now state our main characterization theorem. The proof is given in Appendix~\ref{app:proof}. Recall that we work under an arbitrary fixed measure $\QQ\in\Mcal(\FF)$. We let $\zeta(\T)$ denote the stopping time in~\eqref{eq:zeta(T)} associated with an atomic tree $\T$.

\begin{thm} \label{T:S cont char}
Assume $S$ is continuous. Then semi-static completeness holds if and only if there exists a full atomic tree $\T$ such that
\begin{enumerate}
\item\label{T:S cont char:1} $S$ is complete on $A\times[t(A),T]$ for each leaf $A\in\T$,
\item\label{T:S cont char:2} the set $\{\EE[\psi_i\mid\sigma(\T)] \colon i=1,\ldots,n\}$ contains $\dim\T - 1$ linearly independent elements.
\end{enumerate}
In this case, $L^2(\Fcal_T)= L^2(\sigma(\T))\oplus\Scal(S)$, and $S$ is constant on $\lc0,\zeta(\T)\rc$.
\end{thm}

Semi-static completeness is therefore fully specified by the structure of the filtration. More precisely, a model is semi-statically complete if and only if, under that model, the filtration has the shape depicted in Figure~\ref{F:atomic tree}: there is an atomic component, generated by the part of the statically traded securities that is not replicable by trading in $S$ (see also Lemma~\ref{L:psi_rep}), and a richer component generated by $S$. As a result, a semi-statically complete model is obtained as a combination of dynamically complete models ``glued'' together via an atomic tree.

Clearly such models are ``unphysical'' in the sense that they do not give realistic descriptions of real asset prices. Nonetheless, they do characterize the extreme points of the set of calibrated martingale measures (see Theorem~\ref{T:ext}). Therefore, one may regard Theorem~\ref{T:S cont char} as providing a {\em parameterization} of this set of extreme points in terms of dynamically complete models and atomic trees.

Let us briefly mention how the atomic tree $\T$ arises in the proof of Theorem~\ref{T:S cont char}. The key idea is to consider the unhedgeable parts $V^i_t = \EE[\psi_i\mid\Fcal_t] - (H^i\cdot S)_t$ of the static claims $\psi_i$, where $H^i\cdot S$ is the orthogonal projection of $\EE[\psi_i\mid\Fcal_t]$ onto $\Scal(S)$. Semi-static completeness then implies, by the argument sketched at the end of Section~\ref{Sect:JY}, that
\[
\Scal(V^1,\ldots,V^n) = \vspan\{V^1,\ldots,V^n\}.
\]
This yields a set of atoms via Proposition~\ref{P:SM=spanM}, which are used to construct an atomic tree $\T$ such that $\psi_i=\EE[\psi_i\mid\sigma(\T)]+(H^i\cdot S)_T$; see Lemma~\ref{L:psi_rep}. From this, one deduces~\ref{T:S cont char:1} and~\ref{T:S cont char:2} in a relatively straightforward manner.

\begin{figure}[h]
\begin{center}
\begin{tikzpicture}
    \node (myfirstpic) at (7.5,2.7) {\includegraphics[width=.065\textwidth,height=.06\textwidth]{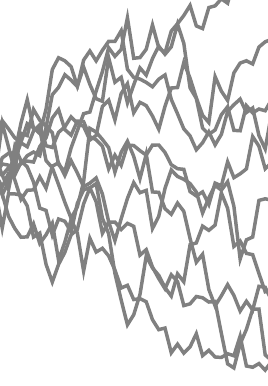}};
    \node (myfirstpic) at (7.5,1.7) {\includegraphics[width=.065\textwidth,height=.06\textwidth]{BM_paths_3.pdf}};
    \node (myfirstpic) at (6.25,-0.61) {\includegraphics[width=.23\textwidth,height=.1\textwidth]{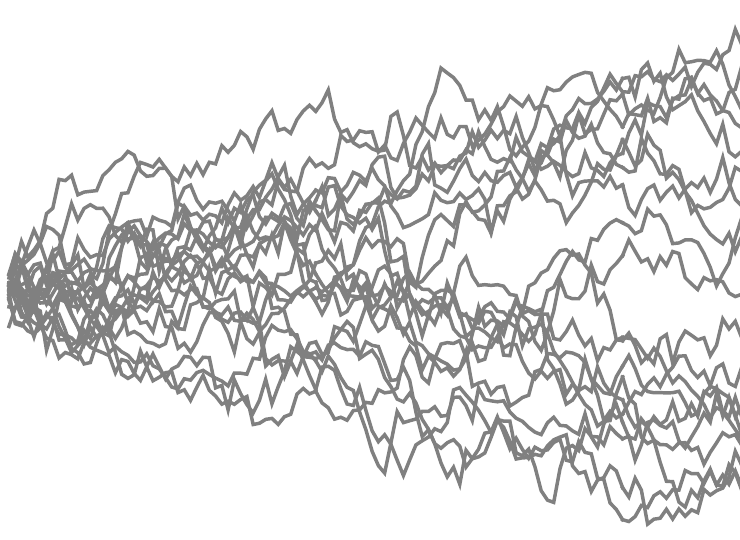}};
    \node (myfirstpic) at (4.97,-2.1) {\includegraphics[width=.4\textwidth,height=.1\textwidth]{BM_paths_2.pdf}};

    \node[left] () at (-0.2, 0) {$\Omega$};
    \node[left] () at (7-0.2, 2.7) {$A_7$};
    \node[left] () at (7-0.2, 1.7) {$A_6$};
    \node[left] () at (4.5-0.2, 0.7) {$A_5$};
    \node[left] () at (4.5-0.2, -0.7) {$A_4$};
    \node[left] () at (2-0.2, -2.2) {$A_1$};
    \node[left] () at (2+0.7, -0.4) {$A_2$};
    \node[left] () at (2-0.2, 2.2) {$A_3$};
    
    \vertex (0) at (0, 0) {};
    \coordinate (10) at (1.7, 0) {};
    \vertex (11) at (2, 2.2) {}; 
    \vertex (12) at (2, 0) {}; 
    \vertex (13) at (2, -2.2) {}; 
    \coordinate (110) at (6.8, 2.2) {}; 
    \vertex (111) at (7, 2.7) {}; 
    \vertex (112) at (7, 1.7) {}; 
    \coordinate (120) at (4.3, 0) {}; 
    \vertex (121) at (4.5, 0.7) {}; 
    \vertex (122) at (4.5, -0.7) {}; 
    \coordinate (1210) at (8, 0.7) {}; 

    \begin{scope}[every path/.style={-}]
       \draw (0) -- (10) ;
       \draw (10) -- (11) ;
       \draw (10) -- (12) ;
       \draw (10) -- (13) ;

       \draw (11) -- (110) ;
       \draw (110) -- (111) ;
       \draw (110) -- (112) ;

       \draw (12) -- (120) ;
       \draw (120) -- (121) ;
       \draw (120) -- (122) ;

       \draw (121) -- node[midway,above] {}  (1210) ;
    \end{scope}
    
    \draw[dotted] (0,-3) -- (0,3);
    \draw[dotted] (2,-3) -- (2,3);
    \draw[dotted] (4.5,-3) -- (4.5,3);
    \draw[dotted] (7,-3) -- (7,3);
    \draw[dotted] (8,-3) -- (8,3);
    \draw[thick,->] (0,-3) -- (0,-3) node[below] {$0$} -- (2,-3) node[below] {$t_1$}  -- (4.5,-3) node[below] {$t_2$}  -- (7,-3) node[below] {$t_3$}  -- (8,-3) node[below] {$T$} -- (9,-3);
\end{tikzpicture}
\end{center}
\caption{Schematic illustration of the filtration $\FF$ when semi-static completeness holds. The wiggly curves emanating from the leaves (except $A_5$) illustrate that the filtration may quickly become rich after~$\zeta(\T)$. It is, however, also possible that no further events occur once a leaf is reached; this is illustrated by the flat line emanating from $A_5$. By semi-static completeness, each of the models starting at the leaves is dynamically complete.}
\label{F:atomic tree}
\end{figure}

\begin{rem}
The tree $\T$ in Theorem~\ref{T:S cont char} is nonempty since it is full. Thus it contains at least one leaf, whence $\dim\T\ge 1$. In the degenerate case where $\T=\{\Omega\}$, \ref{T:S cont char:1} says that $S$ is complete on $\Omega\times[0,T]$, which is simply the usual notion of dynamic completeness. Furthermore, $\T$ is unique up to nullsets. Indeed, if $\T'$ is another possible tree, the theorem implies that $L^2(\sigma(\T))=L^2(\sigma(\T'))$.
\end{rem}

\begin{cor}\label{cor:cont}
Assume $S$ is continuous, let $\QQ\in\ext\Mcal(\FF)$, and let $\T$ denote the associated full atomic tree. Then the jumps of any martingale $M$ are supported on $\T$, in the sense that $\{\Delta M\ne 0\} \subseteq \{ A\times\{t(A)\} \colon A\in\T\}$.
\end{cor}

\begin{proof}
By Theorem~\ref{T:S cont char}, $M=x+V+H\cdot S$ for some $x\in\R$, some martingale $V$ such that $V_T$ is $\Fcal_{\zeta(\T)}$-measurable, and some $H\in L^2(S)$. In particular, $\rc\zeta(\T),T\rc\subseteq\{\Delta M= 0\}$. Next, let $A'\in\T$ be a child of $A\in\T$. Since $\T$ is full, $A$ is an atom of $\Fcal_{t(A')-}$. It is then clear that $A\times(t(A),t(A'))\subseteq\{\Delta M=0\}$. Let $D$ denote the union of all sets of this form and $\rc\zeta(\T),T\rc$. Then $\{\Delta M\ne0\}\subseteq D^c=\{ A\times\{t(A)\} \colon A\in\T\}$, which yields the assertion.
\end{proof}

We conclude the section with two examples. The first example illustrates how Theorem~\ref{T:S cont char} can be used to build semi-statically complete models in a continuous path setting. The second example shows that the statement of Theorem~\ref{T:S cont char} need not be valid if $S$ has jumps.

\begin{exa}
Here we use Theorem~\ref{T:S cont char} to build a semi-statically complete model, putting together two (dynamically) complete models by means of a $2$-dimensional atomic tree. Let $\Omega=C_0([0,T],\R)$ denote the set of real-valued continuous functions on $[0,T]$ vanishing at zero. Let $S$ be the coordinate process, $S_t(\omega)=\omega(t)$, and let $\FF$ be the right-continuous filtration generated by $S$. Let $\Pcal$ be the set of all probability measures on $\Fcal_T$. Assume there is one statically traded security $\psi = [S,S]_T - K$, for some fixed $K>0$.
Fix $t^*\in(0,T)$ and constants $\sigma_1,\sigma_2$ such that $\sigma_1> \sqrt{K/(T-t^*)}>\sigma_2>0$.
We consider two probability measures $\QQ^1, \QQ^2$ on $\Fcal_T$ such that
\[
S_t = \sigma_i W_{t-t^*}\bm 1_{\{t\ge t^*\}} \quad\text{under}\quad \QQ^i,
\]
where $W$ is a standard Brownian motion under $\QQ^i$. We now set $\QQ=\lambda \QQ^1 + (1-\lambda)\QQ^2$, where $\lambda\in(0,1)$ is determined by the calibration condition $\EE_\QQ[\psi]=0$:
\[
0 = \EE_\QQ[\psi] = \lambda\sigma_1^2(T-t^*) + (1-\lambda)\sigma_2^2(T-t^*) - K,
\]
which implies $\QQ\in\Mcal(\FF)$. We let $A_i= \{ \partial^+[S,S]_{t^*} = \sigma_i^2\}$, where $\partial^+$ denotes the right derivative, and note that $\T = \{\Omega,A_1,A_2\}$ is an atomic tree under $\QQ$ with $\dim\T=2$, and such that
\[
\EE_\QQ[\psi\mid\sigma(\T)] = \sigma_1^2(T-t^*)\bm 1_{A_1} + \sigma_2^2(T-t^*)\bm 1_{A_2} - K \not\equiv 0.
\]
Therefore, by Theorem~\ref{T:S cont char}, $\QQ$ is a semi-statically complete model. 
The representation of the corresponding filtration is given in Figure~\ref{F:ExB}.
\end{exa}

\begin{figure}[h]
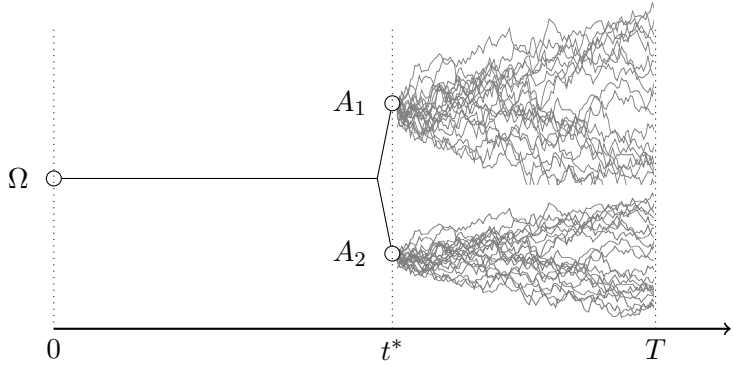

\begin{center}
\begin{tikzpicture}
    \node (myfirstpic) at (6.25,1) {\includegraphics[width=.23\textwidth,height=.2\textwidth]{BM_paths_2.pdf}};
    \node (myfirstpic) at (6.25,-1) {\includegraphics[width=.23\textwidth,height=.12\textwidth]{BM_paths_2.pdf}};

    \node[left] () at (-0.2, 0) {$\Omega$};
    \node[left] () at (4.5-0.2, 1) {$A_1$};
    \node[left] () at (4.5-0.2, -1) {$A_2$};
    
    \vertex (0) at (0, 0) {};
    \coordinate (10) at (1.7, 0) {};
    \coordinate (120) at (4.3, 0) {}; 
    \vertex (121) at (4.5, 1) {}; 
    \vertex (122) at (4.5, -1) {}; 

    \begin{scope}[every path/.style={-}]
       \draw (0) -- node[midway,above] {}  (120) ;
       \draw (120) -- node[midway,above] {}  (121) ;
       \draw (120) -- node[midway,above] {}  (122) ;
    \end{scope}
    
    \draw[dotted] (0,-2) -- (0,2);
    \draw[dotted] (4.5,-2) -- (4.5,2);
    \draw[dotted] (8,-2) -- (8,2);
    \draw[thick,->] (0,-2) node[below] {$0$} -- (4.5,-2) node[below] {$t^*$} -- (8,-2) node[below] {$T$} -- (9,-2);
\end{tikzpicture}
\end{center}
\caption{The leaves $A_1,A_2$ correspond to the two Bachelier models with volatilities $\sigma_1>\sigma_2$. Thus the variance swap $\psi=[S,S]_T-K$ is priced differently under the two models, and can be used to hedge against $A_1$ or $A_2$.
}
\label{F:ExB}
\end{figure}

\begin{exa} \label{ex:2}
Here we provide a semi-statically complete model, for which the filtration structure given in Theorem~\ref{T:S cont char} fails.
We let $\Omega=C_0([0,T],\R)\times\R_+\times\{0,1\}$, and write $(W,\theta,\xi)$ for the coordinate element. Fix $t^*\in(0,T)$ and $\sigma_1,\sigma_2>0$ with $\sigma_1\neq\sigma_2$. The price process $S$ is defined by
\[
S_t = \begin{cases}
-t,	& t<\theta\wedge t^* \\
1-\theta,  & t\ge\theta,\ \theta<t^* \\
-t^* + (\xi\sigma_1 + (1-\xi)\sigma_2) W_{t-t^*},  & t\ge t^*,\ t^*\le \theta,
\end{cases}
\]
and $\FF$ is the right-continuous filtration it generates. Let $\Pcal$ be the set of all probability measures on $\Fcal_T$. Let $\QQ$ be a probability measure on $\Fcal_T$ under which $W$, $\theta$, and $\xi$ are mutually independent, $W$ is a standard Brownian motion, $\theta$ a standard Exponential, and $\xi$ a Bernoulli with parameter $1/2$. Then under $\QQ$, $S$ behaves like a compensated one-jump Poisson process strictly prior to $t^*$. If $S$ jumps, then it stays constant until $T$. Otherwise, if there is no jump before $t^*$, $S$ continues as a Brownian motion whose volatility is either $\sigma_1$ or $\sigma_2$, depending on the outcome of the Bernoulli variable $\xi$.

It is clear that $S$ is a martingale under $\QQ$.
We now introduce the statically traded security $\psi= [S,S]_T - K$ with $K=\EE_\QQ[[S,S]_T]$. As is shown in Lemma~\ref{L:Ex2} below, this makes the model semi-statically complete. However, the filtration $\FF$ does not admit a full atomic tree $\T$ as is guaranteed in the continuous case by Theorem~\ref{T:S cont char}. Thus the statement of the theorem does not carry over to the case where $S$ has jumps.
\end{exa}

\begin{lem} \label{L:Ex2}
The model defined in Example~\ref{ex:2} is semi-statically complete.
\end{lem}

\begin{proof}
Consider any $X\in L^2(\Fcal_T)$ and write
\[
X = \big(X-\EE_\QQ[X|\Fcal_{t^*}]\big) + \EE_\QQ[X|\Fcal_{t^*-}] +
\big(\EE_\QQ[X|\Fcal_{t^*}]-\EE_\QQ[X|\Fcal_{t^*-}]\big).
\]
Using the martingale representation theorem for the Poisson process and Brownian motion, one readily shows that the first two terms on the right-hand side are of the form $\EE_\QQ[X]+(H\cdot S)_T$. To deal with the third term, note that $\Fcal_{t^*}=\Fcal_{t^*-}\vee\sigma(A)$ up to nullsets, where $A=\{\theta\ge t^*\}\cap\{\xi=1\}$ is an atom of $\Fcal_{t^*}$. Thus,
\[
\EE_\QQ[X|\Fcal_{t^*}]-\EE_\QQ[X|\Fcal_{t^*-}] = c\bm 1_A + Y
\]
for some constant $c$ and some $\Fcal_{t^*-}$-measurable random variable $Y$, which admits a representation $Y=y+(J\cdot S)_T$. Thus it remains to show that $\bm 1_A$ can be semi-statically replicated. This follows by taking $X=\psi$ above. Indeed, we have
\[
\psi=\sigma_1^2(T-t^*)\bm1_{A}+\sigma_2^2(T-t^*)\bm1_{B}+\bm1_{C} - K,
\]
where $A$ is as above, $B=\{\theta\ge t^*\}\cap\{\xi=0\}$, and $C=\Omega\setminus (A\cup B)$. Since $A\cup B\in\Fcal_{t^*-}$ and $\sigma_1\ne\sigma_2$, we obtain
\[
\bm 1_A = \frac{1}{(\sigma_1^2-\sigma_2^2)(T-t^*)}\psi + Y'
\]
for some $\Fcal_{t^*-}$-measurable random variable $Y'$, which as above admits a representation in terms of $S$. This concludes the proof of semi-static completeness.
\end{proof}

\section{Pricing by informed investors} \label{S:filter}

In addition to $\FF$, we now consider right-continuous filtrations $\GG=(\Gcal_t)_{0\le t\le T}$ on $(\Omega,\Fcal)$ with $\Fcal_t\subseteq\Gcal_t$ for all $t\le T$. While $\FF$ should be thought of as the information available to a typical market participant, $\GG$ includes additional information that only some investors observe. Notice that $\Mcal(\FF)$ is defined using a family $\Pcal=\Pcal(\FF)$ of probability measures on $\Fcal_T$, while $\Mcal(\GG)$ is similarly defined using a family $\Pcal(\GG)$ of probability measures on $\Gcal_T$. In order to compare the two sets of calibrated martingale measures, we always assume that
\[
\Pcal(\FF) = \{\QQ|_{\Fcal_T}\colon \QQ\in \Pcal(\GG)\}.
\]

For any filtration $\HH=(\Hcal_t)_{0\le t\le T}$, the {\em progressive enlargement} of $\FF$ with $\HH$ is the filtration $\GG=(\Gcal_t)_{0\le t\le T}$ defined by
\begin{equation}\label{eq:GFH}
\Gcal_t = \bigcap_{u>t} \Fcal_u \vee \Hcal_u.
\end{equation}
Thus $\GG$ is the smallest right-continuous filtration that contains both $\FF$ and $\HH$. Our main results consider specifications where $S$ is continuous and $\HH$ is generated by a collection of single-jump processes. By this we mean processes of the form $X\bm 1_{\lc\tau,T\rc}$, where $X$ is a random variable and $\tau$ is a {\em random time}, that is a $[0,T]\cup\{\infty\}$-valued random variable. Without loss of generality we always suppose $\tau=\infty$ on $\{X=0\}$.

To state these results, we define the following $\FF$-stopping time, which is the first time $S$ starts to move:
\[
\sigma = \inf\{t\in[0,T]\colon S_t\ne 0\}.
\]
Moreover, we say that $\FF$ and $\GG$ {\em coincide under $\QQ$} if $\Fcal_t$ equals $\Gcal_t$ up to $\QQ$-nullsets for each $t\in[0,T]$. The following theorem relates semi-static completeness for informed and uninformed investors.

\begin{thm} \label{T:finiteljump}
Assume $S$ is continuous. Let $\GG$ be given by~\eqref{eq:GFH} with $\HH$ generated by finitely many nonnegative bounded single-jump processes $X_i\bm 1_{\lc\tau_i,T\rc}$, $i=1,\ldots,p$. Assume $\tau_i>\sigma$ on $\{0<\tau_i<\infty\}$ for all $i$. Then
\begin{equation}\label{eq:finiteljump}
\ext\Mcal(\GG) = \left\{\QQ\colon \text{$\FF$ and $\GG$ coincide under $\QQ$, and $\QQ\in\ext\Mcal(\FF)$}\right\}.
\end{equation}
\end{thm}

Theorem~\ref{T:finiteljump} can be interpreted as follows: Consider an informed agent who computes super-hedging prices by maximizing over extreme points of $\Mcal(\GG)$ as in~\eqref{eq:supexp}. This agent will find that the relevant models $\QQ$ are those under which the additional information $\HH$ is in fact already contained in $\FF$ (up to nullsets, of course). Example~\ref{ex:tau} below gives a simple illustration of how this restriction can cause $\Mcal(\GG)$ to be significantly smaller than $\Mcal(\FF)$. The difference between these sets yields a potentially large difference between the robust super-hedging prices computed by the informed and uninformed agents. Further examples are given in Section~\ref{sect:ex}.

\begin{exa}\label{ex:tau}
Let $S$ be continuous with $S_0=0$.
Suppose $\HH$ is generated by the single-jump process $\bm 1_{\lc\tau,T\rc}$, where $\tau=\sup\{t\in[0,T]\colon S_t=1\}$ is the last time $S$ hits level~$1$, and where we set $\tau=0$ if this never happens. Note that $\tau$ satisfies the condition in Theorem~\ref{T:finiteljump}. Then, in order for $S$ to be a martingale for $\GG$, we must have $S<1$ almost surely. To see this, observe that a continuous martingale attaining a certain level at a stopping time (in this case, $S_\tau=1$) will return to that level infinitely many times, unless this happens at time $T$. Therefore, either $\tau=T$ or $S<1$ hence $\tau=0$. This implies $\{S_T=1\}=\{\tau=T\}=\{\tau=0\}^c\in \Gcal_0$. Since $S_0=0$, the martingale property imposes $\tau=0$ almost surely, thus forcing $S<1$.
In addition to this property, any $\QQ\in\Mcal(\GG)$ should price the statically traded securities correctly. This example will reappear in Section~\ref{sect:ex}.
\end{exa}

\begin{rem}
Note that the filtration $\GG$ considered in Theorem~\ref{T:finiteljump} is the smallest right-continuous filtration which contains $\FF$, makes the $\tau_i$ stopping times, and the $X_i\bm 1_{\{\tau_i<\infty\}}$ become $\Gcal_{\tau_i}$-measurable. Both the progressive enlargement with a random time and the initial enlargement with a random variable are included as special cases; simply take $X=1$ for the former, and $\tau=0$ for the latter. These classical types of filtration enlargement are the most studied in the literature, and our analysis draws heavily on this theory; see e.g. \cite{JY78a} and \cite{Je80a}.
\end{rem}

\begin{rem}
Due to Theorem~\ref{T:ext}, Theorem~\ref{T:finiteljump} implies that extra information of the form considered here cannot be used to complete the market under a given model. The only way an informed agent can face a semi-statically complete market is when semi-static completeness already holds for the uninformed agent. Therefore, while additional information may reduce the cost of super-replication in an incomplete market, it will in general not be enough to guarantee exact replication.
\end{rem}

\begin{proof}[Proof of Theorem~\ref{T:finiteljump}]
The only inclusion that needs proof is ``$\subseteq$''. In fact, we will prove by induction on $p$ the statement
\begin{equation} \label{eq:5.1:1}
\begin{gathered}
\text{$\ext\Mcal(\GG) \subseteq \{\QQ\colon \text{$\FF'$ and $\GG$ coincide under $\QQ$}\}$ holds} \\
\text{for any right-continuous base filtration $\FF'$, where $\GG$ is the} \\
\text{progressive enlargement of $\FF'$ with $p\ge1$ single-jump processes.}
\end{gathered}
\end{equation}
Suppose for the moment that the base case $p=1$ has been proved. Let $p\ge2$ and assume~\eqref{eq:5.1:1} is true for $p-1$. Let $\HH'$ be the filtration generated by $X_p\bm 1_{\lc\tau_p,T\rc}$, and let $\HH''$ be the filtration generated by $X_i\bm 1_{\lc\tau_i,T\rc}$, $i=1,\ldots,p-1$. Then, with the obvious notation, we have
\[
\GG = \left( \FF \vee \HH' \right) \vee \HH''.
\]
The induction assumption applied with base filtration $\FF'=\FF\vee\HH'$ implies that $\FF \vee \HH'$ and $\GG$ coincide under any $\QQ\in\ext\Mcal(\GG)$. Thus $\Mcal(\GG)=\Mcal(\FF\vee\HH')$. Thus, applying the base case with $\FF'=\FF$ we find that $\FF$ and $\FF\vee\HH'$, and hence $\FF$ and $\GG$, coincide under any $\QQ\in\ext\Mcal(\GG)$. This completes the induction step.

It only remains to prove~\eqref{eq:5.1:1} for the base case where $\HH$ is generated by a single one-jump process $X\bm 1_{\lc\tau,T\lc}$. We write $\FF'=\FF$. Fix a measure $\QQ\in\ext\Mcal(\GG)$. Define a process $M$ by
\begin{equation} \label{eq:JeulinYor_-1}
M_t = X\bm 1_{\{\tau\leq t\}}-\int_0^{t\wedge\tau}\frac1{Z_{s-}}dA_s,
\end{equation}
where $Z$ is the Az\'ema supermartingale~\eqref{eq:A:Azema} associated with $\tau$, and $A$ is the dual predictable projection of the process $X\bm 1_{\lc\tau,\infty\lc}$. By Lemma~\ref{L:JeulinYor}, $M$ is a $\GG$-martingale. A localization argument in conjunction with semi-static completeness, which follows from Theorem~\ref{T:S cont char}, yields
\begin{equation} \label{eq:JeulinYor_0}
M = M_0 + V + H\cdot S
\end{equation}
for some $S$-integrable process $H$ and some $\GG$-martingale $V$ with $V_T\in L^2(\sigma(\T))$, where $\T$ is the corresponding full atomic tree. To see this, first note that $M$ is locally bounded. Indeed, $X$ is bounded, and the integral in~\eqref{eq:JeulinYor_-1} defines a c\`adl\`ag predictable processes which is automatically locally bounded; see VII.32 in~\cite{DM80}. Let $(\rho_k)$ be a localizing sequence. Semi-static completeness and Theorem~\ref{T:S cont char} yield
\[
M^{\rho_k}_T = M_0 + V^k_T + (H^k\cdot S)_T
\]
for some $H^k\in L^2(S,\GG)$ and $V^k_T\in L^2(\Gcal_{\zeta(\T)})$. Since $\rho_k\to\infty$ and $\zeta(\T)\le T$ takes finitely many values, we have $\rho_k>\zeta(\T)$ for all sufficiently large $k$, say $k\ge k_0$. Thus, taking $\Gcal_{\zeta(\T)}$-conditional expectations and using that $S$ is constant on $\lc0,\zeta(\T)\rc$, we have
\[
V^k_T = M_{\zeta(\T)\wedge\rho_k} - M_0 = M_{\zeta(\T)} - M_0
\]
for all $k\ge k_0$. The right-hand side does not depend on $k$; denote it by $V_T$. Then \eqref{eq:JeulinYor_0} holds with $H$ given by
\[
H = H^{k_0}\bm 1_{\lc0,\rho_{k_0}\rc} + \sum_{k>k_0} H^k \bm 1_{\rc \rho_{k-1}, \rho_k\rc},
\]
as claimed.

In view of \eqref{eq:JeulinYor_-1}, \eqref{eq:JeulinYor_0}, and the continuity of~$S$, considering the jump process of $M$ yields
\begin{equation} \label{eq:JeulinYor_1}
X\bm 1_{\lc\tau\rc} = \frac{\Delta A}{Z_-} 1_{\lc0,\tau\rc} + \Delta V =  \left(\frac{\Delta A}{Z_-} + \Delta V\right)1_{\lc0,\tau\rc},
\end{equation}
where $\lc\tau\rc$ denotes the graph of $\tau$, and we use the convention $Y_{0-}=0$ for any process $Y$. Note that $\lc\tau\rc\subseteq\lc0\rc\,\cup\,\rc\sigma,T\rc$ due to the assumption that $\tau>\sigma$ on $\{0<\tau<\infty\}$. Also, $\sigma\ge\zeta(\T)$ since $S$ is constant on $\lc0,\zeta(\T)\rc$. Thus, multiplying both sides of~\eqref{eq:JeulinYor_1} by $\bm 1_{\lc0\rc\,\cup\,\rc\sigma,\infty\lc}$ and using that $\Delta V=0$ outside $\rc0,\zeta(\T)\rc$ by Corollary~\ref{cor:cont}, we obtain
\begin{equation} \label{eq:JeulinYor_2}
X\bm 1_{\lc\tau\rc}  = \frac{\Delta A}{Z_-}\bm 1_{\lc0\rc\,\cup\,\rc\sigma,\tau\rc}.
\end{equation}
Since $X>0$ on $\{\tau<\infty\}$ by assumption, this yields
\[
\tau=\inf\left\{t\in[0,T]\colon \frac{\Delta A}{Z_-}\bm 1_{\lc0\rc\,\cup\,\rc\sigma,T\rc} \ne 0  \right\},
\]
which is an $\FF$-stopping time. It follows that the right- and hence left-hand side of~\eqref{eq:JeulinYor_2} is $\FF$-adapted. Thus the process $X\bm 1_{\lc\tau,T\rc}$ with which we enlarge $\FF$ is already $\FF$-adapted, whence $\FF$ and $\GG$ coincide under $\QQ$, and $\QQ\in\ext\Mcal(\FF)$.
\end{proof}

A slight modification of the proof of Theorem~\ref{T:finiteljump} shows that in the absence of statically traded securities, neither the continuity assumption on $S$, nor the condition on the $\tau_i$, is needed.

\begin{cor}\label{cor:finiteljump}
Assume $\Psi=\emptyset$. Let $\GG$ be given by~\eqref{eq:GFH} with $\HH$ generated by finitely many nonnegative bounded single-jump processes $X_i\bm 1_{\lc\tau_i,T\rc}$, $i=1,\ldots,p$. Then
\[
\ext\Mcal(\GG) = \left\{\QQ\colon \text{$\FF$ and $\GG$ coincide under $\QQ$, and $\QQ\in\ext\Mcal(\FF)$}\right\}.
\]
\end{cor}

\begin{proof}
Again the only non-trivial inclusion is ``$\subseteq$'', and as before it suffices to consider one single-jump process $X\bm 1_{\lc\tau,T\lc}$. Using that (dynamic) completeness holds under any $\QQ\in\ext\Mcal(\GG)$, a similar argument as the one leading to~\eqref{eq:JeulinYor_1} yields
\[
X\bm 1_{\lc\tau\rc} =  \left(\frac{\Delta A}{Z_-} + H \Delta S\right)1_{\lc0,\tau\rc}
\]
for some $H\in L^2(S,\GG)$. Let $J$ be an $\FF$-predictable process with $J\bm 1_{\lc0,\tau\rc}=H\bm 1_{\lc0,\tau\rc}$; see~\eqref{eq:A:H=J}. Replacing $H$ by $J$, one sees as before that $\tau$ is almost surely equal to an $\FF$-stopping time, and then that $X\bm 1_{\lc\tau\rc}$ is already $\FF$-adapted.
\end{proof}

\begin{rem}
Theorem~\ref{T:finiteljump} can be generalized, for example to progressive enlargements with countably many single-jump processes such that, for every $\omega$, only finitely many jumps can occur before $T$. However, some assumption on the enlargement is needed for the conclusion of the theorem to hold. Indeed, let $W$ be a standard Brownian motion under $\QQ$ generating the filtration $\GG$, define $S$ via $S_t = \int_0^t \sgn(W_s)dW_s$, and let $\FF$ be the filtration generated by $S$. Then $S$ is again a Brownian motion, and (dynamic) completeness holds with respect to both filtrations. Thus $\QQ|_{\Fcal_T}\in\ext\Mcal(\FF)$ and $\QQ\in\ext\Mcal(\GG)$, where we take $\Psi=\emptyset$ and $\Pcal(\GG)$ the set of all probability measures on $\Gcal_T$. Nonetheless, it is well-known that $\FF$ coincides with the filtration generated by $|W|$ and is strictly smaller than $\GG$; see Corollary~VI.2.2 in \cite{RY99}. We thank Monique Jeanblanc for pointing this out to us. \end{rem}

\section{Examples}\label{sect:ex}

In this section we provide examples, in discrete and in continuous time, where both sets $\Mcal(\FF)$ and $\Mcal(\GG)$ of calibrated martingales measures are compact (Examples~\ref{ex:fe_d} and~\ref{ex:fe_cts}). In those cases, robust pricing can be done over extreme measures, as in \eqref{eq:supexp}, and the results of Section~\ref{S:filter} can be used to compare pricing for agents with different sets of information. We also give examples where $\Mcal(\FF)$ is compact and $\Mcal(\GG)$ is empty, and where $\Mcal(\FF)$ is compact but $\Mcal(\GG)$ is not (Example~\ref{ex:nocpt}).

\begin{exa}[Discrete time]\label{ex:d}
Suppose $T\in\N$ and let $\Omega=[\underline s,\overline s]^T$ for some constants $\underline s<0<\overline s$. Let $S$ be the piecewise constant interpolation of the coordinate process, $S_t(\omega)=\omega(\lfloor t \rfloor)$, where $\lfloor t \rfloor$ denotes the integer part of $t$ and we set $\omega(0)=0$ by convention. Let $\FF$ be the filtration generated by $S$, and let $\Pcal$ be the set of all probability measures on $\Fcal_T$. Moreover, assume that the payoffs of the statically traded securities are continuous in $\omega$. In this setting Prokhorov's theorem implies that $\Pcal$ is weakly compact. Furthermore, the calibration and martingale restrictions become
\[
\EE_\QQ[\psi] = 0 \qquad\text{and}\qquad \EE_\QQ[ f(S_1,\ldots,S_s) (S_t-S_s) ] =0,
\]
for all integers $0\le s<t\le T$ and all continuous functions $f:\R^s\to\R$. Due to the boundedness of $\Omega$, these constraints are weakly closed. Thus $\Mcal(\FF)\subseteq\Pcal$ is weakly compact.
\end{exa}

\begin{exa}[Filtration enlargement in discrete time]\label{ex:fe_d}
We continue with the setting of Example~\ref{ex:d}. Let $\GG$ be a filtration obtained as the initial enlargement of $\FF$ with a random variable $L=L(S)$ that depends continuously on $S$, that is,
\[
\Gcal_t = \bigcap_{u>t} \Fcal_u \vee \sigma(L),\quad t\in[0,T].
\]
An example of such a random variable is $L(S)=\frac1T\sum_{t=1}^T|S_t|$. The set of martingale measures $\Mcal(\GG)$ then consists of all measures in $\Mcal(\FF)$ that satisfy the condition
\[
\EE_\QQ[f(S_1,\ldots,S_s,L(S)) (S_t-S_s) ] =0,
\]
for all integers $0\le s<t\le T$ and all bounded continuous functions $f:\R^{s+1}\to\R$. By boundedness of $\Omega$ and continuity of $L$, this constraint is weakly closed, and since $\Mcal(\FF)$ is weakly compact, then $\Mcal(\GG)$ is weakly compact as well.

Now, the by Krein-Milman theorem, both $\Mcal(\FF)$ and $\Mcal(\GG)$ can be recovered by their extreme measures, and robust pricing can be done over such measures, see \eqref{eq:supexp}.
The extreme points of $\Mcal(\FF)$ and $\Mcal(\GG)$ are related as in Corollary~\ref{cor:finiteljump}, and this allows us to appreciate the difference in the robust pricing of derivatives by an investor with or without additional information.
\end{exa}

\begin{exa}[Continuous time]\label{ex:c}

Fix $T>0$ and let $\Omega$ be the space $C_0([0,T],\R)$  of real-valued continuous functions on $[0,T]$ vanishing at zero, endowed with the topology of uniform convergence. The price process $S$ is taken to be the coordinate process on $\Omega$, and $\FF$ the right-continuous filtration it generates. We assume that the statically traded securities have payoffs $\psi_i$ that are continuous functions of $\omega$ and satisfy
\begin{equation}\label{eq:bdpsi}
|\psi_i(\omega)|\leq C\big(1+\sup_{t\le T}|\omega(t)|^\kappa\big)
\end{equation}
for some constants $C$ and $\kappa$. Fix $\overline\sigma>0$ and let $\Pcal$ denote the set of probability measures $\PP$ on $\Fcal_T$ such that $S$ is a semimartingale with quadratic variation given by
\begin{equation} \label{exA2:P_new}
[S,S]_t =\int_0^t \sigma_s^2 ds \qquad\text{with}\qquad \sigma_t^2 \le \overline\sigma^2 \quad\text{for all $t\le T$.}
\end{equation}
That is, $S$ has absolutely continuous quadratic variation, and the volatility is bounded by~$\overline\sigma$. This situation is discussed for instance in \cite{STZ11}.

\begin{lem} \label{L:M(F) cpct}
$\Mcal(\FF)$ is weakly compact.
\end{lem}

\begin{proof}
We first show that $\Mcal(\FF)$ is weakly closed, so let $\QQ_k\in\Mcal(\FF)$ converge weakly to some probability measure $\QQ$. The BDG inequality and~\eqref{exA2:P_new} yield $\EE_{\QQ_k}[ \sup_{t\le T} |S_t|^p ] \le C_p \overline\sigma^pT^{p/2}$ for any $p\ge 1$, where $C_p$ is a constant that only depends on~$p$. Since the right-hand side is uniform in $k$, Theorem~3.5 in~\cite{Bi13} asserts that $\EE_{\QQ_k}[X]\to\EE_\QQ[X]$ holds for any continuous random variable $X$ that grows at most polynomially in $\sup_{t\le T}|\omega(t)|$. This immediately yields $\EE_\QQ[|S_t|]<\infty$ for all $t$, and in view of~\eqref{eq:bdpsi} also $\EE_\QQ[\psi_i^2]<\infty$ and $\EE_\QQ[\psi_i] = 0$. Moreover, one obtains $\EE_\QQ[(S_t-S_s)X]=0$ for any $\sigma(S_u:u\le s)$-measurable bounded continuous random variable $X$. A monotone class argument lets us drop continuity of $X$, showing that $S$ is a $\QQ$-martingale for the filtration $(\sigma(S_u\colon u\le t))_{t\in[0,T]}$. This is extended to the right-continuous modification $\FF$ by dominated convergence; we omit the details. Thus $S$ is a $\QQ$-martingale for the filtration $\FF$. It only remains to check that $\QQ\lll\Pcal$. Since $\QQ_k\lll\Pcal$ and hence $\QQ_k\in\Pcal$, Lemma~\ref{L:DM^2 bound} yields
\[
\EE_{\QQ_k}\Big[ \Big( \sum_{i=1}^m (S_{t_i}-S_{t_{i-1}})^2 \Big)^p \Big] \le \overline\sigma^{2p}(t-s)^p \left( 1 +  \frac{4^p p!}{m}\right)
\]
for any equidistant grid $0\le s=t_0<\cdots<t_m=t\le T$ and any $p\ge1$ with $p<m$. By the same weak convergence argument as above, this bound carries over to $\QQ$. Considering the grid points $t_i^m=s+i(t-s)/m$ and using Fatou's lemma, we therefore obtain
\[
\EE_\QQ\Big[ \Big( [S,S]_t-[S,S]_s \Big)^p \Big] \le \lim_{m\to\infty} \EE_\QQ\Big[ \Big( \sum_{i=1}^m (S_{t_i^m}-S_{t_{i-1}^m})^2 \Big)^p \Big] \le \overline\sigma^{2p}(t-s)^p
\]
for every $p\ge 1$. Taking $p$th roots of the left- and right-hand sides and sending $p$ to infinity finally yields $[S,S]_t-[S,S]_s\le \overline\sigma^2(t-s)$. This shows that $\QQ\in\Pcal$, and completes the proof that $\Mcal(\FF)$ is closed.

It remains to prove that $\Mcal(\FF)$ is tight. To this end, let $p>2$ and fix any $\varepsilon>0$. Then, the Markov and BDG inequalities together with~\eqref{exA2:P_new} yield
\[
\frac{1}{\delta}\QQ\Big(\sup_{t\leq s\leq t+\delta} |S_s-S_t|\geq\varepsilon\Big)
\le \frac{1}{\delta \varepsilon^p}\EE_\QQ\Big[ \sup_{t\leq s\leq t+\delta} |S_s-S_t|^p\Big]
\le  \frac{C_p\overline\sigma^p}{\varepsilon^p} \delta^{p/2-1}
\]
for any $\delta>0$ and any $t\in[0,T]$ (we set $S_s=S_T$ for $s>T$). By shrinking $\delta$, the right-hand side can be made arbitrarily small. Tightness now follows from Theorem~7.3 and the subsequent Corollary in \cite{Bi13}.
\end{proof}
\end{exa}

\begin{exa}[Filtration enlargement in continuous time]\label{ex:fe_cts}
We continue with the setting of Example~\ref{ex:c}. Consider a filtration $\GG$ obtained as the progressive enlargement of $\FF$ with a random time $\tau$ that depends continuously on $\omega\in C_0([0,T],\R)$. Then an element $\QQ\in\Mcal(\FF)$ lies in $\Mcal(\GG)$ if and only if the martingale condition
\[
\EE_\QQ[f(S_{t_1},\ldots,S_{t_k},\tau\wedge s)(S_t-S_s)]=0
\]
holds for all $0\le t_1<\cdots<t_k\le s\le t\le T$ and all bounded continuous functions $f:\R^k\times[0,T]\to\R$. By continuity of $\tau$, this is a closed condition. Hence $\Mcal(\GG)$ is a weakly closed subset of $\Mcal(\FF)$, thus weakly compact due to Lemma~\ref{L:M(F) cpct}. The relevant question, therefore, is whether $\Mcal(\GG)$ is nonempty.

To give a simple concrete example, suppose $\tau(\omega)=f(\omega(T))$ for some continuous function $f:\R\to[0,T]$ that is equal to zero on an interval $[-a,a]$ around the origin. Then, for any model $\QQ\in\Mcal(\FF)$ such that $|S_T|\le a$ almost surely, we have $\tau=0$ almost surely, hence $\FF$ and $\GG$ coincide under $\QQ$, and thus $\QQ\in\Mcal(\FF)$. Therefore, provided that the condition $|S_T|\le a$ is consistent with the calibration conditions $\EE_\QQ[\psi_i]=0$, $\Mcal(\GG)$ will be nonempty. It will however typically be significantly smaller than $\Mcal(\FF)$, which is advantageous to an informed agent computing robust super-hedging prices. Again, for both agents pricing can be done over extreme measures, which for $\tau$ satisfying Theorem~\ref{T:finiteljump} are related via~\eqref{eq:finiteljump}.
\end{exa}

\begin{exa}[Filtration enlargement and arbitrage or failure of compactness]\label{ex:nocpt}
There are natural enlargements under which $\Mcal(\GG)$ fails to be compact. For example, let $\GG$ be the progressive enlargement of $\FF$ with the hitting time $\tau=\sup\{t\in[0,T]\colon S_t=1\}$ from Example~\ref{ex:tau}, with $\tau=0$ if this set is empty. As explained in Example~\ref{ex:tau}, in order for $S$ to be a martingale for $\GG$ we must have $S<1$ almost surely. This is however not a closed condition. Indeed, if $\QQ_k$ is the law of Brownian motion on $[0,T]$ stopped the first time it hits $1-k^{-1}$, then $\QQ_k$ converges weakly to the law, $\QQ$ say, of Brownian motion stopped the first time it hits~$1$. But $S<1$ fails under $\QQ$, so $\QQ\notin\Mcal(\GG)$.

Furthermore, depending on the static claims, $\Mcal(\GG)=\emptyset$ can occur. For example, consider the case where there is only one static claim $\psi=(S_T-1)_+-\frac12$. The pricing condition  $\EE_\QQ[\psi]=0$ imposes $\QQ(S_T > 1)>0$, hence in this case $\Mcal(\GG)=\emptyset$. This situation is interpreted as existence of arbitrage for the informed agent.
\end{exa}

\appendix

\section{A technical lemma}\label{A:QV bound}

\begin{lem} \label{L:DM^2 bound}
Fix $\overline\sigma>0$ and let $M$ be a continuous martingale with $M_0=0$ and $[M,M]_t-[M,M]_s \le \overline\sigma^2 (t-s)$ for all $0\le s\le t\le T$. Then, for any equidistant grid $0\le s=t_0<\cdots<t_m=t\le T$ and any $p>0$ with $p<m$, we have
\[
\EE\Big[ \Big( \sum_{i=1}^m (M_{t_i}-M_{t_{i-1}})^2 \Big)^p \Big] \le \overline\sigma^{2p}(t-s)^p\left( 1 + \frac{4^p p!}{m} \right).
\]
\end{lem}

\begin{proof}
The proof uses the double factorial defined by $n!!=n\times (n-2) \cdots 3\times 1$ for $n$ odd, and $(-1)!!=1$.
We claim that
\begin{equation} \label{eq:moment bound 1}
\EE[M_t^{2k}] \le (2k-1)!!\, \overline\sigma^{2k} t^k, \qquad 0\le t\le T,
\end{equation}
holds for any $k\ge 0$. We proceed by induction on $k$. For $k=0$, the statement is obviously true. Let now $k\ge1$ and assume~\eqref{eq:moment bound 1} is true for $k-1$. It\^o's lemma yields
\[
M_t^{2k} = 2k \int_0^t M_s^{2k-1}dM_s + k (2k-1) \int_0^t M_s^{2k-2} d[M,M]_s.
\]
The local martingale term is a true martingale due to the bound on $[M,M]$; we omit the argument here. Taking expectations and using the induction assumption as well as the bound on $[M,M]$, this yields
\begin{align*}
\EE[M_t^{2k}] &\le \overline\sigma^2 k (2k-1) \int_0^t \EE[ M_s^{2k-2} ] ds \\
&\le  \overline\sigma^{2k} k (2k-1) (2k-3)!! \int_0^t s^{k-1} ds \\
&=  \overline\sigma^{2k} t^k (2k-1)!!
\end{align*}
as required. Thus~\eqref{eq:moment bound 1} holds for all $k$ by induction.

Next, for ease of notation write $\Delta M_i = M_{t_i}-M_{t_{i-1}}$ and $h=t_i-t_{i-1}=(t-s)/m$. A conditioning argument in conjunction with~\eqref{eq:moment bound 1} yields, for any nonnegative integers $k_1,\ldots,k_m$,
\[
\EE[ \Delta M_1^{2k_1} \cdots \Delta M_m^{2k_m} ] \le (\overline\sigma^2 h)^{k_1+\cdots+k_m} \prod_{i=1}^m (2k_i - 1)!!
\]
Combining this with the multinomial theorem yields
\begin{align*}
\EE\Big[ \Big( \sum_{i=1}^m \Delta M_i^2 \Big)^p \Big] &= \sum_{k_1+\cdots+k_m=p}{p\choose k_1\, \cdots\, k_m} \EE[ \Delta M_1^{2k_1} \cdots \Delta M_m^{2k_m} ] \\
&\le \overline\sigma^{2p}h^p \sum_{k_1+\cdots+k_m=p}{p\choose k_1\, \cdots\, k_m} \prod_{i=1}^m (2k_i - 1)!! \\
&= \overline\sigma^{2p}h^p m^p + \overline\sigma^{2p}h^p \sum_{k_1+\cdots+k_m=p}{p\choose k_1\, \cdots\, k_m} \left( \prod_{i=1}^m (2k_i - 1)!! - 1\right) \\
&\le  \overline\sigma^{2p}h^p m^p + 4^p p!\,\overline\sigma^{2p}h^pm^{p-1},
\end{align*}
where the last line uses the combinatorial inequality~\eqref{eq:multinomial ineq} below. Since $mh=t-s$, the result follows.

It remains to prove the inequality
\begin{equation}\label{eq:multinomial ineq}
\sum_{k_1+\cdots+k_m=p}{p\choose k_1\, \cdots\, k_m} \left( \prod_{i=1}^m (2k_i - 1)!! - 1\right) \le 4^p p!\, m^{p-1}.
\end{equation}
We proceed by induction, noting first that \eqref{eq:multinomial ineq} holds for $p=1$ since the left-hand side is zero in this case. We now suppose \eqref{eq:multinomial ineq} holds for $p$ and prove it for $p+1$. Since any multi-index $(k_1,\ldots,k_m)$ summing to $p+1$ can be represented in at least one way as $(l_1,\ldots,l_{j-1},l_j+1,l_{j+1},\ldots,l_m)$ for some $j$ and some multi-index $(l_1,\ldots,l_m)$ summing to $p$, we have
\begin{align*}
&\sum_{k_1+\cdots+k_m=p+1}{p+1\choose k_1\, \cdots\, k_m} \left( \prod_{i=1}^m (2k_i - 1)!! - 1\right)  \\
&\qquad \le \sum_{j=1}^m \sum_{l_1+\cdots+l_m=p}{p\choose l_1\, \cdots\, l_m} \frac{p+1}{l_j+1} \left( (2l_j+1) \prod_{i=1}^m (2l_i - 1)!! - 1\right).
\end{align*}
This expression equals
\[
\sum_{j=1}^m \sum_{l_1+\cdots+l_m=p}{p\choose l_1\, \cdots\, l_m} \frac{p+1}{l_j+1} \left( (2l_j+1) \Big( \prod_{i=1}^m (2l_i - 1)!! - 1\Big) + 2l_j\right),
\]
which, since \eqref{eq:multinomial ineq} is assumed to hold for $p$, is bounded by
\[
2(p+1) 4^p p! m^p + 2(p+1) \sum_{j=1}^m l_j \sum_{l_1+\cdots+l_m=p}{p\choose l_1\, \cdots\, l_m} = 2(p+1)(4^p p! + p) m^p.
\]
The right-hand side is crudely bounded by $4^{p+1}(p+1)!m^p$, showing that \eqref{eq:multinomial ineq} holds for $p+1$ as desired.
\end{proof}

\section{Proof of Theorem~\ref{T:S cont char}} \label{app:proof}

In proving Theorem~\ref{T:S cont char} we treat sufficiency and necessity separately. Sufficiency is fairly straightforward, so we deal with this first. The last two statements of the theorem will follow in the course of the proof; see~\eqref{eq:X=T+HS} below regarding the direct sum decomposition of $L^2(\Fcal_T)$, and Corollary~\ref{C:S const on 0zeta(T)} for the constancy of $S$. Recall that $\QQ\in\Mcal(\FF)$ is fixed.

\begin{proof}[Proof of Theorem~\ref{T:S cont char}: sufficiency]
Let $\T$ be a full atomic tree satisfying \ref{T:S cont char:1}--\ref{T:S cont char:2}. We need to prove that any $X\in L^2(\Fcal_T)$ admits a semi-static representation.

To start with, we claim that any $X\in L^2(\Fcal_T)$ has a representation
\begin{equation} \label{eq:X=T+HS}
X = \EE[X\mid\sigma(\T)] + (H\cdot S)_T
\end{equation}
for some $H\in L^2(S)$. To prove this, let $A_1,\ldots,A_d$ denote the leaves of $\T$. Since $\T$ is full, the leaves form a partition of $\Omega$ (up to a nullset). Together with the assumption that $S$ is complete on $A_i\times[t(A_i),T]$ for each $i$, this yields
\[
X = \sum_{i=1}^d X\bm 1_{A_i} = \sum_{i=1}^d \left( x_i + ( H^i \cdot S)_T \right) \bm 1_{A_i}
\]
for some $x_i\in\R$ and some $H^i\in L^2(S)$ with $H^i=0$ on $\lc0,t(A_i)\rc$. Defining
\[
H=\sum_{i=1}^d H^i\bm 1_{A_i},
\]
we then have $H\in L^2(S)$ and $H=0$ on $\lc0,\zeta(\T)\rc$. Thus
\[
X = \sum_{i=1}^d x_i\bm 1_{A_i} + (H\cdot S)_T
\]
and, using~\eqref{eq:zeta(T)} and the optional stopping theorem,
\[
\EE[X\mid\sigma(\T)]=\EE[X\mid\Fcal_{\zeta(\T)}] = \sum_{i=1}^d x_i\bm 1_{A_i}.
\]
We deduce~\eqref{eq:X=T+HS}, as desired.

It now suffices to prove that any $\sigma(\T)$-measurable random variable $X$ (which is automatically bounded, hence square-integrable) admits a semi-static representation. In view of~\eqref{eq:X=T+HS}, we can find $H^i\in L^2(S)$ such that
\begin{equation} \label{eq:psi_i=EsTHiS}
\psi_i = \EE[\psi_i\mid\sigma(\T)] + (H^i\cdot S)_T, \qquad i=1,\ldots,n.
\end{equation}
Due to assumption~\ref{T:S cont char:2} and Remark~\ref{R:atomic tree}\ref{R:atomic tree:4}, the constant $1$ together with the random variables $\EE[\psi_i\mid\sigma(\T)]$ span $L^2(\sigma(\T))$. Thus there exist constants $a_0,\ldots,a_n$ such that
\[
X = a_0 + \sum_{i=1}^n a_i\EE[\psi_i\mid\sigma(\T)].
\]
In conjunction with~\eqref{eq:psi_i=EsTHiS} this yields
\[
a_0 + \sum_{i=1}^n a_i \psi_i = a_0 + \sum_{i=1}^n a_i \EE[\psi_i\mid\sigma(\T)] + \sum_{i=1}^n a_i (H^i\cdot S)_T = X + (H\cdot S)_T,
\]
where $H=\sum_{i=1}^n a_iH^i$ lies in $L^2(S)$. Thus $X$ admits a semi-static representation, as required. This completes the proof of sufficiency in Theorem~\ref{T:S cont char}.
\end{proof}

To prove the forward implication (necessity) of Theorem~\ref{T:S cont char} we need some preliminary results. These results, specifically Proposition~\ref{P:SM=spanM}, Lemmas~\ref{L:semimgFV} and~\ref{L:contMconst}, and Corollary~\ref{C:S const on 0zeta(T)} below, do not use the fact that $\QQ\in\Mcal(\FF)$. Indeed they are valid for any filtered probability space $(\Omega,\Fcal,\FF,\QQ)$ whose filtration $\FF$ is right-continuous.  We write $\Fcal_{0-}=\Fcal_0$ by convention.

\begin{prop} \label{P:SM=spanM}
Let $M=(M^1,\ldots,M^d)$ be a $d$-dimensional square-integrable weakly orthonormal martingale with $\EE[M_0]=0$; in particular, we assume that $\EE[M^i_TM^j_T]=\delta_{ij}$ (the Kronecker delta) for $i,j=1,\ldots,d$. Assume also that
\begin{equation} \label{L:SM=spanM:1}
\Scal(M) = \vspan\{M^1,\ldots,M^d\}.
\end{equation}
Then there exists an orthogonal matrix $Q\in O(d)$ and time points $0\le t_1<\cdots<t_m\le T$, such that the martingale $N=QM$ is of the form
\begin{equation}\label{eq:NiNiT}
N = \begin{pmatrix}N^{(1)} \\ \vdots \\ N^{(m)}\end{pmatrix} = \begin{pmatrix}N^{(1)}_T \bm 1_{\lc t_1,T\rc} \\ \vdots \\ N^{(m)}_T \bm 1_{\lc t_m,T\rc}\end{pmatrix},
\end{equation}
where each $N^{(k)}$ is a $d_k$-dimensional martingale for some $1\le d_k\le d$. Each martingale $N^{(k)}=(N^{(k),1},\ldots,N^{(k),d_k})$ satisfies
\begin{equation} \label{L:SNi=spanNi:1}
\Scal(N^{(k)}) = \vspan\{N^{(k),1},\ldots,N^{(k),d_k}\}.
\end{equation}
Moreover, for each $k$ there exist $d_k$ pairwise disjoint atoms $B^k_1,\ldots,B^k_{d_k}$ of $\Fcal_{t_k-}$ such that
\begin{equation} \label{eq:Ni_atoms}
\{ \QQ(N^{(k)}_T \ne 0 \mid\Fcal_{t_k-}) > 0\} = B^k_1 \cup \cdots \cup B^k_{d_k}.
\end{equation}
\end{prop}

\begin{rem}
Of course, $d_1+\cdots+d_m=d$. Furthermore, note that some of the atoms $B^k_i$ may be nullsets.
\end{rem}

\begin{proof}
We first prove the existence of $Q\in O(d)$ and $0\le t_1<\cdots<t_m\le T$ such that $N=QM$ is of the form~\eqref{eq:NiNiT}. We proceed by induction on $d$. The case $d=0$ is vacuously true, so we need only prove the result under the assumption that it holds with $d$ replaced by any $r<d$.

For any fixed $t\in[0,T]$ and $i\in\{1,\ldots,d\}$, define a bounded predictable process $H=(H^1,\ldots,H^d)$ by $H^i=\bm 1_{\lc 0,t\rc}$ and $H^j=0$ for $j\ne i$. Then~\eqref{L:SM=spanM:1} implies
\begin{equation} \label{eq:MtCMT}
M^i_t = (H\cdot M)_T = \sum_{j=1}^d C^{ij}_t M^j_T
\end{equation}
for some deterministic constants $C^{ij}_t$. Let $C_t$ denote the matrix with elements $C^{ij}_t$. We claim that
\begin{equation} \label{eq:Cij_t_0}
\text{$C_t$ is symmetric and the map $t\mapsto C_t$ is c\`adl\`ag.}
\end{equation}
Indeed, orthonormality of $M$, \eqref{eq:MtCMT}, and the martingale property of $M$ yield
\begin{equation} \label{eq:Cij_t}
C^{ij}_t = \sum_{k=1}^d C^{ik}_t\, \EE[M^k_TM^j_T] = \EE[M^i_tM^j_T] = \EE[M^i_tM^j_t] = \EE[ [M^i,M^j]_t],
\end{equation}
showing that $C_t$ is symmetric. Moreover, the Kunita-Watanabe inequality followed by the Cauchy-Schwartz inequality and square-integrability of $M$ yield
\[
|[M^i,M^j]_t|\le [M^i,M^i]_t^{1/2}[M^j,M^j]_t^{1/2} \le [M^i,M^i]_T^{1/2}[M^j,M^j]_T^{1/2} \in L^1.
\]
Thus, since $[M^i,M^j]$ is c\`adl\`ag, \eqref{eq:Cij_t} and the dominated convergence theorem imply~\eqref{eq:Cij_t_0}.

Now, define
\[
t_m = \inf\{t \ge0 : \rk C_t = d\}.
\]
Since $C_T=\id$ due to~\eqref{eq:Cij_t}, we have $t_m\in[0,T]$. Since $C_t$ is c\`adl\`ag, we have $\rk C_{t_m}=d$. If $t_m=0$, then~\eqref{eq:MtCMT} yields $M_T = C_0^{-1}M_0$, which is $\Fcal_0$-measurable. Thus $M$ is constant, and~\eqref{eq:NiNiT} holds with $m=1$ and $Q$ the identity. If instead $t_m>0$, there exists some $s\in[0,t_m)$ such that $\rk C_t$ is constant on $[s,t_m)$. Let $r=\rk C_s<d$ be this constant value.

Since $C_s$ is symmetric, we have $C_s=Q_1^\top \Lambda Q_1$ for some $Q_1\in O(d)$ and some diagonal matrix $\Lambda=\Diag(\lambda_1,\ldots,\lambda_r,0,\ldots,0)$, where $\lambda_1,\ldots,\lambda_r$ are nonzero. Define a new orthonormal martingale $\widehat M$ and matrix map $\widehat C$ by
\[
\widehat M = Q_1 M \qquad \text{and} \qquad \widehat C=Q_1CQ_1^\top.
\]
Due to~\eqref{L:SM=spanM:1}, we have
\begin{equation} \label{eq:SMhat=spanMhat:1}
\Scal(\widehat M) = \vspan\{\widehat M^1,\ldots,\widehat M^d\}.
\end{equation}
Moreover, in view of~\eqref{eq:MtCMT},
\[
\widehat M_t = \widehat C_t \widehat M_T \qquad \text{for all $t\in[0,T]$}.
\]
We establish some properties of $\widehat C$. Since $\rk\widehat C_{t_m}=d$, we have $\widehat M_T=\widehat C_{t_m}^{-1}\widehat M_{t_m}$, which is $\Fcal_{t_m}$-measurable. Thus $\widehat M$ is constant on $[t_m,T]$, whence $\widehat C_t=\id$ there. Next, we have by construction that $\widehat C_s$ is diagonal with $\widehat C_s^{ii}=\lambda_i\ne0$ for $i=1,\ldots,r$. Thus $\widehat M^{i}_T=\lambda^{-1}_i\widehat M^{i}_s$ is $\Fcal_s$-measurable for $i=1,\ldots,r$, so that $\widehat M^{i}$ is constant on $[s,T]$. It follows that
\[
\widehat C^{ij}_t=\sum_{k=1}^d \widehat C^{ik}_t\, \EE[\widehat M^k_T\widehat M^j_T] =\EE[\widehat M^i_t\widehat M^j_T]=\EE[\widehat M^i_T\widehat M^j_T]=\delta_{ij}
\]
for all $t\in[s,T]$, $i\in\{1,\ldots,r\}$, and $j\in\{1,\ldots,d\}$. Since $\rk \widehat C_t=r$ for $t\in[s,t_m)$, this forces $\widehat C^{ij}_t=0$ for $i,j\in\{r+1,\ldots,d\}$ and $t\in[s,t_m)$. To summarize, we have shown that
\[
\widehat C_t = \begin{pmatrix} \id & 0 \\ 0 & \id\, \bm 1_{[t_m,T]}(t)\end{pmatrix} \in \S^{r+(d-r)} \quad\text{for}\quad t\in[s,T].
\]
Defining $\widehat M'=(\widehat M^1,\ldots,\widehat M^r)$ and $\widehat M''=(\widehat M^{r+1},\ldots,\widehat M^d)$, it follows that
\begin{equation} \label{eq:Mi1j}
\widehat M''=\widehat M''_T\bm 1_{\lc t_m,T\rc},
\end{equation}
and that $\widehat M'$ is constant on $[s,T]$. This immediately implies that $\widehat M'$ and $\widehat M''$ are strongly orthogonal. Thus, for any $H\in L^2(\widehat M')$, we have that $(H\cdot \widehat M')_T$ is orthogonal to $\widehat M^i_T$ for $i=r+1,\ldots,d$. Consequently in the representation
\[
(H\cdot \widehat M')_T = a_1 \widehat M^1_T + \cdots + a_d \widehat M^d_T,
\]
which exists by~\eqref{eq:SMhat=spanMhat:1}, we actually have $a_{r+1}=\cdots=a_d=0$. This proves the non-trivial inclusion in
\begin{equation} \label{L:SM=spanM:3}
\Scal(\widehat M') = \vspan\{\widehat M^1,\ldots,\widehat M^r\}.
\end{equation}
Since $\widehat M'$ itself is an $r$-dimensional weakly orthonormal martingale with $\EE[\widehat M'_0]=0$, we may now apply the induction assumption to get $Q_1'\in O(r)$ and time points $0\le t_1< \cdots < t_{m-1}\le T$ such that $N'=Q_1'\widehat M'=(N^1,\ldots,N^r)$ satisfies
\begin{equation}\label{eq:Ni1j}
N' = \begin{pmatrix}N^{(1)} \\ \vdots \\ N^{(m-1)}\end{pmatrix}, \qquad\text{where}\qquad N^{(k)}=N^{(k)}_T \bm 1_{\lc t_k,T\rc},\qquad k=1,\ldots,m-1.
\end{equation}
Since $\widehat M'$, and hence $N'$, is constant on $[s,T]$, we in fact have $t_{m-1}< t_m$. Defining the matrix
\[
Q = Q_2Q_1 \quad\text{where}\quad Q_2 = \begin{pmatrix} Q_1' & 0 \\ 0 & \id \end{pmatrix} \in O(d),
\]
it follows from \eqref{eq:Mi1j} and \eqref{eq:Ni1j} that $N=QM=(N^{(1)},\ldots,N^{(m-1)},\widehat M'')$ is of the desired form. This completes the proof of~\eqref{eq:NiNiT}.

Next, \eqref{L:SNi=spanNi:1} follows by the same argument that gave~\eqref{L:SM=spanM:3}, using the obvious fact that the martingales $N^{(k)}$ are mutually strongly orthogonal.

Finally, for each $k\in\{1,\dots,m\}$ we locate $d_k$ pairwise disjoint atoms $B^k_1,\ldots,B^k_{d_k}$ of $\Fcal_{t_k-}$ such that~\eqref{eq:Ni_atoms} holds (recall our convention $\Fcal_{0-}=\Fcal_0$). To this end, define
\[
B^k = \{ \QQ(N^{(k)}_T \ne 0 \mid\Fcal_{t_k-}) > 0\} \in \Fcal_{t_k-}.
\]
We need to decompose $B^k$ into $d_k$ atoms. Consider any bounded $\Fcal_{t_k-}$-measurable random variable $h$ vanishing outside $B^k$. Define the predictable process $H=h\bm 1_{\lc t_k,T\rc}$. Due to~\eqref{eq:NiNiT} and~\eqref{L:SNi=spanNi:1}, we then have
\[
hN^{(k),i}_T = (H\cdot N^{(k),i})_T = \sum_{j=1}^{d_k} a_{ij} N^{(k),j}_T, \qquad i=1,\ldots,d_k,
\]
for some constants $a_{ij}$. Note that this holds also when $t_k=0$, due to our convention regarding the time-zero value of stochastic integrals. Letting $A$ be the $d_k\times d_k$ matrix with elements $a_{ij}$, we write this in vector form as
\[
h N^{(k)}_T = A N^{(k)}_T.
\]
This implies that $h$ is an eigenvalue of $A$ on $B^k$. To see this, note that
\[
\{\text{$h$ is an eigenvalue of $A$}\} \supseteq \{N^{(k)}_T\ne 0\}.
\]
Taking $\Fcal_{t_k-}$-conditional probabilities and using that $h$ is $\Fcal_{t_k-}$-measurable, we get
\[
\bm 1_{\{\text{$h$ is an eigenvalue of $A$}\}} \ge \QQ(N^{(k)}_T \ne 0 \mid\Fcal_{t_k-}) > 0 \quad \text{on} \quad B^k,
\]
by definition of $B^k$. This shows that $h$ is an eigenvalue of $A$ on $B^k$. Since the $d_k\times d_k$ matrix $A$ can have at most $d_k$ distinct eigenvalues, the random variable $h$ can take at most $d_k$ different values. Since $h$ was arbitrary, we deduce the existence of a decomposition
\[
B^k = B^k_1\cup \cdots \cup B^k_{d_k}
\]
into (possibly trivial) atoms of $\Fcal_{t_k-}$, as required.
\end{proof}

\begin{rem}
The conclusions~\eqref{eq:NiNiT} and~\eqref{eq:Ni_atoms} of Proposition~\ref{P:SM=spanM} are not quite strong enough to imply~\eqref{L:SM=spanM:1}. More specifically, the decomposition~\eqref{eq:Ni_atoms} into atoms is not enough to imply \eqref{L:SNi=spanNi:1}; one also has to account for the interplay between the processes $N^{(k),i}$ on each atom $B^k_j$. This is captured by the relation
\[
\vspan\{N^{(k),1},\ldots,N^{(k),d_k}\} = \vspan\{ \bm 1_{B^k_j}N^{(k),i} : i,j = 1,\ldots,d_k\},
\]
which is a consequence of~\eqref{L:SNi=spanNi:1} and~\eqref{eq:Ni_atoms}. It is not hard to show that together with~\eqref{eq:NiNiT} and~\eqref{eq:Ni_atoms}, this actually implies~\eqref{L:SM=spanM:1}, yielding a converse of Proposition~\ref{P:SM=spanM}. As the current formulation of Proposition~\ref{P:SM=spanM} is sufficient for our purposes, we refrain from developing this line of reasoning further.
\end{rem}

\begin{lem} \label{L:semimgFV}
Let $X$ be a semimartingale with $X_t = f(t)$ for all $t<\tau$, where $f$ is a deterministic function and $\tau$ is a stopping time. Then $f$ is of finite variation on $[0,t^*]$ for any $t^*$ such that $\QQ(\tau\ge t^*)>0$.
\end{lem}

\begin{proof}
Suppose for contradiction that $f$ is not of finite variation on $[0,t^*]$. Then there exist partitions $0=t_0^n<\cdots<t_{N_n}^n=t^*$ indexed by $n$ such that $\sum_{i=1}^{N_n} |f(t^n_i) - f(t^n_{i-1})|\to\infty$ as $n\to\infty$. For each $n$, define the elementary predictable process $H^n=\widetilde H^n\bm 1_{\lc 0,\tau\rc}$, where $\widetilde H^n=\sum_{i=1}^{N_n}\sgn(f(t^n_i)-f(t^n_{i-1}))\bm 1_{\rc t^n_{i-1},t^n_i\rc}$. Then
\[
(H^n\cdot X)_{t^*} = \sum_{i=1}^{N_n} |f(\tau\wedge t^n_i) - f(\tau\wedge t^n_{i-1})|,
\]
whence $\QQ(\liminf_n (H^n\cdot X)_{t^*} = \infty) \ge \QQ(\tau\ge t^*) >0$. Since $X$ is a semimartingale and $|H^n|\le 1$, this gives the desired contradiction. Indeed, as a direct consequence of the bounded convergence theorem for stochastic integrals, the set
\[
\big\{ (H\cdot X)_{t^*} \colon \textrm{$H$ is elementary predictable with $|H|\le 1$}\big\}
\]
is bounded in probability.
\end{proof}

\begin{lem} \label{L:contMconst}
Let $M$ be a continuous local martingale, and let $B$ be an atom of $\Fcal_{t^*}$ or of $\Fcal_{t^*-}$ for some $t^*>0$. Then $M_t=M_0$ on $B$ for all $t<t^*$.
\end{lem}

\begin{proof}
Consider the stopping time $\tau=\inf\{t\colon \QQ(B\mid\Fcal_t)=0\}$ as well as the events $B_t=\{\tau>t\}=\{\QQ(B\mid\Fcal_t)>0\}\supseteq B$ for $t<t^*$. We have $\tau=\infty$ on $B$, and we may suppose that $\QQ(B)>0$. Now, suppose for contradiction that $B_t = B_1\cup B_2$ for two disjoint non-nullsets $B_1,B_2\in\Fcal_t$. Then $B\subseteq B_1$ (possibly after relabeling) since $B$ is an atom, and thus $\QQ(B\mid\Fcal_t)=0$ on $B_2$, a contradiction. Hence $B_t$ is an atom of $\Fcal_t$, which implies $M_t\bm 1_{B_t} = f(t)\bm 1_{B_t}$ for some $f(t)\in\R$. Thus $M_t=f(t)$ for all $t<\tau\wedge t^*$, so Lemma~\ref{L:semimgFV} yields that $f$ is of finite variation on $[0,t^*]$. Thus $M^{\tau\wedge t^*}$ is a continuous local martingale of finite variation, and therefore constant. This completes the proof.
\end{proof}

\begin{cor} \label{C:S const on 0zeta(T)}
Let $\T$ be a full atomic tree and $M$ a continuous local martingale. Then $M$ is constant on $\lc0,\zeta(\T)\rc$.
\end{cor}

\begin{proof}
Lemma~\ref{L:contMconst} implies that $M$ is constant on $A\times[0,t(A))$ for each leaf $A\in\T$. Thus $M$ is constant on $\lc0,\zeta(\T)\lc$, and the result follows by continuity of $M$.
\end{proof}

The following result is the key step toward proving the forward implication of Theorem~\ref{T:S cont char}. This is where the required full atomic tree is constructed. Once this has been done, it is straightforward to complete the proof of the theorem.

\begin{lem} \label{L:psi_rep}
Assume $S$ is continuous and semi-static completeness holds. Then there exists a full atomic tree $\T$ such that each $\psi_i$, $i=1,\ldots,n$, admits a representation
\[
\psi_i = \EE[\psi_i\mid\sigma(\T)] + (H^i\cdot S)_T
\]
for some $H^i\in L^2(S)$.
\end{lem}

\begin{proof}
If dynamic completeness holds, the result is clearly true with $\T=\{\Omega\}$. We therefore suppose that dynamic completeness fails. For each $i=1,\ldots,n$, let $H^i\cdot S$ be the orthogonal projection of the martingale $\EE[\psi_i\mid\Fcal_t]$ onto $\Scal(S)$ and define the martingale $V=(V^1,\ldots,V^n)$ by
\[
V^i_T = \psi_i - (H^i\cdot S)_T, \qquad i=1,\ldots,n.
\]
Suppose we can find a full atomic tree $\T$ such that
\begin{equation} \label{eq:ViTsTmb}
\text{$V^i_T$ is $\sigma(\T)$-measurable for $i=1,\ldots,n$.}
\end{equation}
Then, since $S$ is constant on $\lc0,\zeta(\T)\rc$ by Corollary~\ref{C:S const on 0zeta(T)}, we have $\EE[(H^i\cdot S)_T\mid\Fcal_{\zeta(\T)}]=0$. Since also $\sigma(\T)=\Fcal_{\zeta(\T)}$ up to nullsets, we deduce
\[
\psi_i - (H^i\cdot S)_T = \EE[\psi_i \mid\sigma(\T)],
\]
which is the required conclusion.

We thus only need to find a full atomic tree $\T$ such that~\eqref{eq:ViTsTmb} holds. To this end, note that each $V^i$ is weakly orthogonal to $\Scal(S)$, and hence also strongly orthogonal; see e.g.~Theorem~VIII.49 in \cite{DM80}. Together with semi-static completeness and the fact that $\EE[V^i_T]=0$, this yields
\[
\vspan\{V^1,\ldots,V^n\} = \Scal(S)^\perp \supseteq \Scal(V),
\]
where $\Scal(S)^\perp$ denotes the weak orthogonal complement of $\Scal(S)$. Since $\vspan\{V^1,\ldots,V^n\}\subseteq\Scal(V)$ holds trivially, we actually have equality, and this suggests that Proposition~\ref{P:SM=spanM} should be used. To prepare for this, choose a (weakly) orthonormal martingale $M=(M^1,\ldots,M^d)$ with $\vspan\{M^1,\ldots,M^d\}=\vspan\{V^1,\ldots,V^n\}$. Here $d=\dim\vspan\{V^1,\ldots,V^n\}$, and we have $d\ge1$ since we assumed that dynamic completeness does not hold. The martingale $M$ inherits the property
\[
\Scal(M) = \vspan\{M^1,\ldots,M^d\}.
\]

Proposition~\ref{P:SM=spanM} now yields $Q\in O(d)$ and $0\leq t_1<\cdots<t_m\le T$, $m\ge1$, such that the martingale $N=QM$ satisfies \eqref{eq:NiNiT}--\eqref{eq:Ni_atoms} for some atoms $B^k_1,\ldots,B^k_{d_k}$ of $\Fcal_{t_k-}$, with $d_1+\cdots+d_m=d$. Since $V$, $M$, and $N$ are related by (deterministic) invertible linear transformations, \eqref{eq:ViTsTmb} is equivalent to
\begin{equation} \label{eq:NiTsTmb}
\text{$N^i_T$ is $\sigma(\T)$-measurable for $i=1,\ldots,d$,}
\end{equation}
and semi-static completeness means that we have
\begin{equation} \label{eq:L2SN1Nd}
L^2(\Fcal_T)= \vspan\{1,N^1,\ldots,N^d\}\oplus\Scal(S).
\end{equation}

We now construct $\T$ inductively. Set $k=1$ and $\T=\{\Omega\}$. With $N^{(0)}=0$, it is clear that the pair $(k,\T)$ satisfies the following induction hypothesis:
\begin{equation} \label{eq:ind hyp key lemma}
\begin{array}{l}
\text{$\T$ is full,}\\
\text{$\Fcal_{t(A)}\subseteq\Fcal_{t_k-}$ for all $A\in\T$,}\\
\text{$N^{(l)}$ is $\sigma(\T)$-measurable for $l<k$.}
\end{array}
\end{equation}

Assume now that the pair $(k,\T)$ satisfies~\eqref{eq:ind hyp key lemma}. The induction step proceeds as follows. First, the case $k=1$, $t_1=0$ requires special treatment. In this case, we simply re-define $\T$ to consist of those events among the $\Fcal_0$-measurable atoms $B^1_1,\ldots,B^1_{d_1},\Omega\setminus(B^1_1\cup\ldots\cup B^1_{d_1})$ that are non-null. If $m=1$, we are done: \eqref{eq:NiTsTmb} holds. Otherwise, set $k=2$, note that \eqref{eq:ind hyp key lemma} is satisfied for the new pair $(\T,k)$, and proceed with the induction.

Next, consider the case $t_k>0$. Equation \eqref{eq:Ni_atoms} says that $\{\QQ(N^{(k)}_T\ne0\mid\Fcal_{t_k-})>0\}=B^k_1\cup\cdots\cup B^k_{d_k}$. Consider $B^k_1$. Either $\QQ(B^k_1)=0$, in which case we ignore it, or $\QQ(B^k_1)>0$. In the latter case, since $\T$ is full, we can find a leaf $A$ with $\QQ(A\cap B^k_1)>0$. We now show that then, in fact,
\begin{equation} \label{eq:A=Bk1}
\text{$A=B^k_1$ up to a nullset.}
\end{equation}
To this end, first observe that $B^k_1\subseteq A$ since $B^k_1$ is an atom and $A\in\Fcal_{t(A)}\subseteq\Fcal_{t_k}$ by~\eqref{eq:ind hyp key lemma}. Next, using~\eqref{eq:L2SN1Nd} and taking $\Fcal_{t_k-}$-conditional expectations while keeping in mind~\eqref{eq:NiNiT}, we obtain
\begin{equation} \label{eq:1Bk1etc}
\bm 1_{B^k_1} = \QQ(B^k_1) + (H\cdot S)_{t_k} + \sum_{l<k} a_l^\top N^{(l)}_T
\end{equation}
for some $H\in L^2(S)$ and some $a_l\in\R^{d_l}$, $l<k$. Since $A$ is a leaf of $\sigma(\T)$, and using the induction hypothesis~\eqref{eq:ind hyp key lemma}, we find
\[
\sum_{l<k} a_l^\top N^{(l)}_T = c \quad \text{on}\quad A
\]
for some constant $c$. Moreover, Lemma~\ref{L:contMconst} implies that $(H\cdot S)_{t_k}=0$ on $B^k_1$. Inspecting~\eqref{eq:1Bk1etc} on the event $B^k_1$ we thus obtain
\[
\QQ(B^k_1) + c = 1.
\]
Next, another application of Lemma~\ref{L:contMconst} yields
\[
\bm 1_A(H\cdot S)_{t_k} = \bm 1_A(H\bm 1_{\rc t(A),t_k\rc}\cdot S)_{t_k} = (K\cdot S)_{t_k},
\]
where $K=H\bm 1_A \bm 1_{\rc t(A),t_k\rc}\in L^2(S)$. Thus, multiplying both sides of~\eqref{eq:1Bk1etc} by $\bm 1_A$ we get
\[
\bm 1_{B^k_1} = \bm 1_A\left( \QQ(B^k_1) + c \right) + (K\cdot S)_{t_k} = \bm 1_A + (K\cdot S)_{t_k}.
\]
Taking expectations yields $\QQ(B^k_1)=\QQ(A)$. Together with the fact that $B^k_1\subseteq A$, this proves~\eqref{eq:A=Bk1}.

Repeating this for each $B^k_i$, we identify events $A_1,\ldots,A_p$ that are leaves of $\T$ and atoms of $\Fcal_{t_k-}$ and satisfy
\begin{equation} \label{eq:Nkn0A1Ap}
\{\QQ(N^{(k)}_T\ne0\mid\Fcal_{t_k-})>0\}=A_1\cup\cdots\cup A_p.
\end{equation}
On this set we have $S_t=S_0$ for $t\le t_k$ due to Lemma~\ref{L:contMconst}. Together with~\eqref{eq:L2SN1Nd} and in view of~\eqref{eq:NiNiT}, this implies that the linear space
\[
\{X\in L^2(\Fcal_{t_k})\colon X=0 \text{ outside } A_1\cup\cdots\cup A_p\}
\]
is spanned by $\{N^{(l),1},\ldots,N^{(l),d_l}\colon l=1,\ldots,k\}$ together with $\bm 1_{A_1\cup\cdots\cup A_p}$. In particular, it is finite-dimensional. Thus each set $A_i$ can be decomposed into finitely many non-null atoms of $\Fcal_{t_k}$, which we denote by $A_{ij}$, $j=1,\ldots,n_i$, that satisfy
\[
\sigma(N^{(k)}) \subseteq \sigma(A_{ij}\colon i=1,\ldots,p,\, j=1,\ldots, n_i)
\]
up to nullsets. Moreover, since $N^{(k)}$ is a martingale and not identically zero on $A_i$ due to~\eqref{eq:Nkn0A1Ap}, and since $t_k>0$, we have $n_i\ge 2$ for each~$i$. Define the atomic tree
\[
\T'=\T\cup\{A_{ij}\colon i=1,\ldots,p,\, j=1,\ldots, n_i\}.
\]
The above observations together with the induction hypothesis~\eqref{eq:ind hyp key lemma} show that $\T'$ is again full, $N^{(1)},\ldots,N^{(k)}$ are $\sigma(\T')$-measurable, and $t(A)<t_{k+1}$ for all $A\in\T'$ (setting $t_{m+1}=\infty$). Now replace $\T$ by $\T'$. If $k=m$, we are done: \eqref{eq:NiTsTmb} holds. Otherwise, we replace $k$ by $k+1$, observe that~\eqref{eq:ind hyp key lemma} is satisfied for the new pair $(k,\T)$, and iterate. The procedure ends after $m$ steps. The proof is complete.
\end{proof}

We can now complete the proof of Theorem~\ref{T:S cont char}.

\begin{proof}[Proof of Theorem~\ref{T:S cont char}: necessity]
Let $\T$ be the full atomic tree given by Lemma~\ref{L:psi_rep}. We first prove~\ref{T:S cont char:1}. Let $A$ be any leaf of $\T$, and consider an arbitrary $X\in L^2(\Fcal_T)$. By semi-static completeness,
\[
X = a_0 + \sum_{i=1}^n a_i\psi_i + (H\cdot S)_T
\]
for some constants $a_0,\ldots,a_n$ and some $H\in L^2(S)$. By Lemma~\ref{L:psi_rep} we have $H^1,\ldots,H^n$ in $L^2(S)$ such that
\begin{equation} \label{eq:XxKSA_0}
X = a_0 + \sum_{i=1}^n a_i\EE[\psi_i\mid\sigma(\T)] + (K\cdot S)_T,
\end{equation}
where $K=H+H^1+\cdots+H^n$. Since any $\sigma(\T)$-measurable random variable is constant on $A$, we have
\begin{equation} \label{eq:XxKSA}
X = x + (K\cdot S)_T \quad\text{on}\quad A
\end{equation}
for some $x\in\R$. Finally, Lemma~\ref{L:contMconst} shows that $S$ is constant on $A\times[0,t(A)]$. We may thus replace $K$ by $K\bm 1_{\rc t(A),T\rc}$ without invalidating~\eqref{eq:XxKSA}, and conclude that $S$ is complete on $A\times[t(A),T]$. This proves~\ref{T:S cont char:1}.

We now prove~\ref{T:S cont char:2}. Again, let $A$ be any leaf of $\T$. With $X=\bm 1_A$, \eqref{eq:XxKSA_0} yields
\[
\bm 1_A = a_0 + \sum_{i=1}^n a_i\EE[\psi_i\mid\sigma(\T)] + (K\cdot S)_T,
\]
where, by Corollary~\ref{C:S const on 0zeta(T)}, we have $(K\cdot S)_{\zeta(\T)}=0$. The optional stopping theorem and the fact that $\sigma(\T)=\Fcal_{\zeta(\T)}$ up to nullsets then yield
\[
\bm 1_A = a_0 + \sum_{i=1}^n a_i\EE[\psi_i\mid\sigma(\T)].
\]
We conclude that $\{\EE[\psi_i\mid\sigma(\T)] \colon i=1,\ldots,n\}$ together with the constant $1$ span the $(\dim\T)$-dimensional space $L^2(\sigma(\T))$, and since $\EE[\psi_i]=0$ for all $i$ the former set must contain $\dim\T - 1$ linearly independent elements. This proves~\ref{T:S cont char:2}.
\end{proof}

\section{A Jeulin-Yor theorem}

In this section we state and prove a slight generalization of the classical Jeulin-Yor theorem from the theory of progressive enlargement of filtrations; see \cite{JY78a} and \cite{GZ08}, among others. This result is needed in Section~\ref{S:filter}.

Recall that we work on a given filtered measurable space $(\Omega,\Fcal,\FF)$ whose filtration $\FF$ is right-continuous. Let $\tau$ be a random time and $X$ a nonnegative bounded random variable such that $\tau=\infty$ on $\{X=0\}$. Denote by $\HH$ the filtration generated by the single-jump process $X\bm 1_{\lc \tau,T\rc}$. Define $\GG$ as the progressive enlargement of $\FF$ with $\HH$; see~\eqref{eq:GFH}. Then in particular,
\begin{equation}\label{eq.FcaltGetc}
\Fcal_t\cap\{\tau>t\} = \Gcal_t\cap\{\tau>t\} \quad\text{for all $t\in[0,T]$},
\end{equation}
see Lemma~2.5 in~\cite{KLP13}. By (the proof of) Lemma~1 in \cite{JY78a} this implies that for any $\GG$-predictable process $H$ there exists an $\FF$-predictable process $J$ such that
\begin{equation}\label{eq:A:H=J}
J\bm 1_{\lc0,\tau\rc}=H\bm 1_{\lc0,\tau\rc}.
\end{equation}
Next, fix any probability measure $\QQ$ on $\Gcal_T$. Let $Z$ denote the right-continuous supermartingale associated with~$\tau$ by Az\'ema \cite{Az72} via
\begin{equation}\label{eq:A:Azema}
Z_t = \QQ(\tau>t\mid\Fcal_t),
\end{equation}
and let $A$ denote the dual predictable projection of the bounded process $X\bm 1_{\lc\tau,\infty\lc}$; see Theorem~12 in Appendix~I of~\cite{DM80}. Note that the usual conditions are not assumed and not needed here. By Lemma~A.10 in \cite{GZ08}, whose proof still goes through in our setting, we have $Z_->0$ except on a $dA$-nullset.

\begin{lem} \label{L:JeulinYor}
The process $M$ given by
\begin{equation*}\label{eq:L:JeulinYor}
M_t = X \bm 1_{\{\tau\leq t\}}-\int_0^{t\wedge\tau}\frac1{Z_{s-}}dA_s
\end{equation*}
is a $\GG$-martingale.
\end{lem}

\begin{proof}
We follow the proof of the Jeulin-Yor theorem given by \cite{GZ08}. Define $N=X\bm 1_{\lc\tau,\infty\lc}$, let $H$ be any bounded $\GG$-predictable process, and let $J$ be an $\FF$-predictable process satisfying~\eqref{eq:A:H=J}. Note that $Z_-$ coincides with the predictable projection of $\bm 1_{\lc0,\tau\rc}$ by the same argument as in the proof of Theorem~1.1 in \cite{GZ08}. Using also the definition of $A$, we then obtain
\begin{align*}
\EE\left[ \int_0^\infty H_t dN_t \right] &= \EE\left[ H_\tau X\bm 1_{\{\tau<\infty\}} \right] = \EE\left[ J_\tau X\bm 1_{\{\tau<\infty\}} \right] = \EE\left[ \int_0^\infty J_t dN_t \right] = \EE\left[ \int_0^\infty J_t dA_t \right] \\
&= \EE\left[ \int_0^\infty J_t Z_{t-}\frac{dA_t}{Z_{t-}}  \right] = \EE\left[ \int_0^\infty J_t \bm 1_{\lc0,\tau\rc}\frac{dA_t}{Z_{t-}} \right] = \EE\left[ \int_0^\infty H_t \bm 1_{\lc0,\tau\rc}\frac{dA_t}{Z_{t-}} \right].
\end{align*}
This proves the lemma.
\end{proof}

\bibliography{Acciaio_Larsson}{}
\bibliographystyle{alpha}

\end{document}